\documentclass[a4paper,11pt]{article}
\usepackage{amsmath,amssymb,amsthm}
\usepackage{hyperref,url,enumitem,bm,esint}
\usepackage{color}
\usepackage[margin=30mm]{geometry}

\definecolor{rosso}{rgb}{0,0,0}

\def\anold #1{#1}
\def\an #1{{\color{rosso}#1}}
\def\todo #1{{\color{black}#1}}
\newcommand{\vp}{\varphi}
\newcommand{\si}{\sigma}
\newcommand{\nx}{\nabla}
\newcommand{\Lx}{\Delta}
\newcommand{\W}{\mathcal{W}}
\newcommand{\E}{\mathcal{E}}
\newcommand{\pdnu}{\pd_{\bm{n}}}
\renewcommand{\div}{\mathrm{div}}
\renewcommand{\u}{\bm{u}}
\newcommand{\bE}{\bar{\mathcal{E}}}
\newcommand{\C}{\mathcal{C}}
\newcommand{\0}{\bm{0}}
\newcommand{\GD}{\Gamma_D}
\newcommand{\GN}{\Gamma_N}
\newcommand{\SD}{\Sigma_D}
\newcommand{\SN}{\Sigma_N}
\newcommand{\g}{\bm g}
\newcommand{\bn}{\bm n}
\newcommand{\R}{\mathbb{R}}
\newcommand{\bet}{\bm{\eta}}
\newcommand{\no}[1]{\| #1 \|}
\newcommand{\dx}{\, dx}
\newcommand{\dt}{\, dt}

\newcommand{\dz}{\, dz}
\newcommand{\pd}{\partial}
\newcommand{\X}{X (\Omega)}
\newcommand{\inn}[2]{\langle #1 , #2 \rangle}
\newcommand{\J}{\mathcal{J}}
\newcommand{\U}{\mathcal{U}}
\newcommand{\bS}{\mathcal{S}}
\newcommand{\bY}{\mathcal{Y}^\beta}
\newcommand{\bh}{\bm{h}}
\newcommand{\bv}{\bm{v}}
\newcommand{\bYl}{\mathcal{Y}^\beta_{\mathrm{lin}}}
\newcommand{\bX}{\mathcal{X}^\beta}
\newcommand{\bz}{\bm{z}}
\newcommand{\bs}{\bm{s}}
\newcommand{\bZ}{\mathcal{Z}^\beta}
\newcommand{\G}{\mathcal{G}}
\newcommand{\K}{\mathcal{K}}
\newcommand{\bH}{\mathcal{H}}

\newcommand{\PP}{\mathbb{P}}
\newcommand{\eps}{\varepsilon}
\newcommand{\N}{\mathbb{N}}

\theoremstyle{plain}
\newtheorem{thm}{Theorem}[]

\newtheorem{cor}{Corollary}[section]
\newtheorem{remark}{Remark}[section]

\numberwithin{equation}{section}

\title{Sparse optimal control of a phase field tumour model with mechanical effects}

\author{Harald Garcke \footnotemark[1] \and Kei Fong Lam \footnotemark[2]
\and Andrea Signori \footnotemark[3]}
\date{ }

\begin{document}

\maketitle

\renewcommand{\thefootnote}{\fnsymbol{footnote}}
\footnotetext[1]{Fakult{\"a}t f\"ur Mathematik, Universit{\"a}t Regensburg, 93040 Regensburg, Germany
({\tt Harald.Garcke@mathematik.uni-regensburg.de}).}
\footnotetext[2]{Department of Mathematics, Hong Kong Baptist University, Kowloon Tong, Hong Kong ({\tt akflam@math.hkbu.edu.hk}).}
\footnotetext[3]
{Dipartimento di Matematica e Applicazioni, Universit\`a di Milano--Bicocca
via Cozzi 55, 20125 Milano, Italy ({\tt andrea.signori02@universitadipavia.it}).}
\begin{abstract}
In this paper, we study an optimal control problem for a macroscopic mechanical tumour model based on the phase field approach.  The model couples a Cahn--Hilliard
\an{type} equation to a system of linear elasticity and a reaction-diffusion equation for a nutrient concentration.  By taking advantage of previous analytical well-posedness results established by the authors, we seek optimal controls in the form of a boundary nutrient supply, as well as concentrations of cytotoxic and antiangiogenic drugs that minimise a cost functional involving mechanical stresses.  Special attention is given to sparsity effects, where with the inclusion of convex non-differentiable regularisation terms to the cost functional, we can infer from the \anold{first-order} optimality conditions that the optimal drug concentrations can vanish on certain time intervals.\end{abstract}



\section{Introduction}
Mechanical stresses \an{play} a significant role in both enhancing and inhibiting the growth of tumours.  The unregulated proliferation of tumour cells displaces nearby normal tissues and in turn these tissues exert externally applied stress to resist tumour expansion.  In various experimental studies (see \cite{Cheng,Helm,Jain,Sty} and the references cited therein) high compressive stress has the effect of suppressing proliferation and can induce apoptosis (natural cell death) in tumour cells.  However, in the case where the mechanical loads are not uniform, tumours can adapt by growing in directions of least stress.  Moreover, deformations of the microenvironment brought about by these mechanical loads can alter the structure of nearby blood and lymphatic vessels, which are responsible for supplying the region with crucial nutrients, oxygen, therapeutic drugs, as well as drainage of excessive interstitial fluids containing waste products.  The gradual reduction in blood flow turns the stressed region more hypoxic and more acidic, compounding with the reduction in nutrients levels further accelerates the invasive and metastatic potentials of the tumours cells.  On the other hand this also impair\an{s} the effectiveness of immune cells or therapeutic agents as they are not able to reach certain tumour regions in sufficient quantities.  With the use of mathematical modelling \cite{Katira,StyJ}, treatments aimed at alleviating stress seem to be a promising avenue that warrant further investigations and could be use\an{d} in coordination with other anti-cancer therapies.

Recent progress in mathematical oncology have shown promising results in forecasting tumour growth and predictive simulations of treatments \cite{Agosti1,Agosti2,Lima1,Lima2,Lor1,Lor2}.  Most models employ a continuum description involving partial differential equations to capture a multitude of biological and chemical mechanisms.  Among those we focus on the subclass of phase field tumour models \cite{CrisBook,GLSS,Oden,Wise}, where the corresponding numerical simulations (\an{see, e.g.,} \cite{Cris1,Cris2,Frieboes,Wise,Xu}) are able to replicate commonly observed morphologies exhibited by tumours and their vasculatures. 

While there has been a surge of activity in the subsequent mathematical modelling and analysis of phase field tumour models, see \cite{Agosti1,ColliGomez,Cris1,Cris2,GLSS,Oden,Xu} and the references cited therein, there seems to comparatively fewer focus on mechanical interactions in tumour growth within the subclass of phase field models, aside from recent contributions \cite{Faghihi,GLS1,Lima1,Lima2}, \todo{see also \cite{BCDGSS,CMA,Gela,GelaSing} for results concerning the related Cahn--Larch\'e system}.  In light of the significance of mechanical stress, for our study, we consider a simplification of the phase field model that was proposed and studied in the authors' previous work \cite{GLS1}.  Consider a bounded domain $\Omega \subset \R^d$, $d \in \{2,3\}$, with boundary $\Gamma := \pd \Omega$ that is either of $C^{1,1}$-regularity or is convex
\anold{and is partitioned into two subregions $\GD$ and $\GN$.}
For an arbitrary $T > 0$ (which can be interpreted as the length of the medical treatment), the following model posed in the space-time cylinder \anold{$\Omega \times (0,T)$} describes the evolution of a cellular mixture containing tumour and non-tumour cells subject to various mechanisms involving a chemical species acting as nutrient and mechanical stresses:
\begin{subequations}\label{state}
\begin{alignat}{3}
\vp_t & = \Lx \mu + U(\vp, \si, \E(\u)) && \text{ in } Q := \Omega \times (0,T), \label{s:vp} \\
\mu & = - \Lx \vp + \Psi'(\vp) - \chi \sigma + \W_{,\vp}(\vp,\E(\u)) && \text{ in } Q, \label{s:mu} \\
\W_{,\vp}(\vp, \E(\u)) & =  -\C(\E(\u) - \bE - \vp \E^*)  : \E^*&& \\
\beta \si_t & = \Lx \si + S(\vp, \si) && \text{ in } Q, \label{s:si} \\
\anold{\0} & = \div (\W_{,\E}(\vp, \E(\u))) && \text{ in } Q, \label{s:E} \\
\W_{,\E}(\vp,\E(\u)) & = \C(\E(\u) - \bE - \vp \E^*) && \\
\vp(0) & = \vp_0 , \quad \si(0) = \si_0 && \text{ in } \Omega, \label{ic} \\
0 & = \pdnu \vp = \pdnu \mu, \quad \pdnu \si + \kappa \anold{(\si- \si_B) =0} && \text{ on } \Sigma := \Gamma \times (0,T), \label{bc1} \\
\u & = \0 && \text{ on } \Sigma_D := \Gamma_D \times (0,T), \label{bc2} \\
\W_{,\E}(\vp, \E(\u)) \bn & = \g && \text{ on } \Sigma_N := \Gamma_N \times (0,T)\anold{.} \label{bc3}
\end{alignat}
\end{subequations}
We refer the reader to \cite{GLS1,Lima1,Lima2} for more background on the model and related topics, while briefly describe the main components.  In the above, the variable $\vp$ denotes a phase field parameter that serves to distinguish between the two different types of cellular material in the mixture, with tumour cells occupying the region $\{\vp = 1\}$ and non-tumour cells in the region $\{\vp = -1\}$.  The subsystem \eqref{s:vp}-\eqref{s:mu} constitutes a Cahn--Hilliard type equation, where $\mu$ is the associated chemical potential. \anold{Coupled to this is a reaction-diffusion equation \eqref{s:si} for a nutrient $\sigma$, as well as 
a quasistatic linear elasticity system \eqref{s:E} with displacement $\u$ and symmetric strain tensor $\E(\u) := \frac{1}{2}( \nabla \u + (\nabla \u)^{\top})$.}  We mention that there are certain cases where the nutrient evolves quasistatically, which is covered by the case $\beta = 0$.  The terms $\W_{,\vp}(\vp,\E(\u))$ and $\W_{,\E}(\vp, \E(\u))$ are partial derivatives of the elastic energy $\W(\vp, \E(\u))$ with respect to its first and second arguments, respectively, and for this work we consider the choice
\begin{align*}
\W(\vp, \E(\u)) = \frac{1}{2}(\E(\u) - \bE - \vp \E^*) : \C (\E(\u) - \bE - \vp \E^*),
\end{align*}
where $\C$ is a constant, symmetric and positive definite elasticity tensor satisfying the usual symmetry conditions, and the phase-dependent stress-free strain $\bE(\vp)$ under Vegard's law is given by the linear ansatz $\bE(\vp) = \bE + \vp \E^*$ with constant symmetric second order tensors $\bE$ and $\E^*$.  Furthermore, in \eqref{s:mu} the directed movement of cells by chemotaxis is captured by the term $- \chi \sigma$, with $\chi \geq 0$ playing the role of chemotactic sensitivity \cite{GLSS}, while the term $\Psi'(\vp)$ is the derivative of a double-well potential $\Psi(\vp)$ with equal minima at $\vp = \pm 1$.  In our setting this term plays the role of cellular adhesion that leads to the development of regions of tumour and non-tumour cells well-separated by interfacial layers described by the set $\{ -1 < \vp < 1 \}$. 

For boundary conditions we subdivide the boundary $\Gamma$ into the partition
\begin{align*}
\Gamma = \overline{\Gamma_D} \cup \overline{\Gamma_N} \quad \text{ such that } \quad \Gamma_D \cap \Gamma_N = \emptyset.
\end{align*}
Both portions are assumed to be relatively open and to have positive Hausdorff measures, and on the portion $\Gamma_D$, representing a rigid structure of the tumour environment such as bone, the displacement $\u$ is set to be zero, and on the complement portion $\Gamma_N$, the normal component of the stress tensor $\W_{,\E}$ is equal to some given load $\g$ provided by a fixed source.  Meanwhile, \eqref{bc1} highlights that the cellular diffusive flux $\pdnu \mu$ is zero across the boundary, and for $\kappa > 0$ the nutrient flux $\pdnu \si$ is proportional to the difference between a nutrient source $\si_B$ from nearby capillaries and the nutrient level at the boundary.  The case of a zero nutrient diffusive flux is covered by the choice $\kappa = 0$.  

Lastly, the source term $U(\vp, \si, \E(\u))$ in \eqref{s:vp} captures cellular growth that can be influenced by nutrient concentration and mechanical stress.  The example we will use is
\begin{align*}
U(\vp, \si, \E(\u)) = \lambda_p \si f(\vp) g(\W_{,\E}(\vp,\E(\u))) - (\lambda_a + m(t)) k(\vp),
\end{align*}
where $\lambda_p \geq 0$, $\lambda_a \geq 0$ are constant proliferation and apoptosis (cell death) rates, and $f$, $g$ and $k$ are Lipschitz and bounded functions.  For instance, we can model the proliferation and apoptosis of tumour cells only by prescribing the conditions $f(1) = k(1) = 1$, $f(-1) = k(-1)= 0$, \an{see, e.g.,}~\cite{GLSS}.  Meanwhile, to account for the effect of reduced proliferation due to increase in mechanical stress \cite{Byrne,Cheng,Helm,Sty}, we may consider as a motivating example the function $g: \R^{d \times d} \to \R$ defined as
\begin{align}\label{eg:g}
g(\bm{A}) = \frac{1}{\sqrt{1+|\bm{A}|^2}} \text{ for } \bm{A} \in \R^{d \times d},
\end{align}
where $|\bm{A}|$ is the Frobenius norm of the matrix $\bm{A}$, so that as the magnitude of the stress $\W_{,\E}(\vp, \E(\u))$ increases, the effects of the proliferation term $\lambda_p \si f(\vp) g(\W_{,\E}(\vp, \E(\u)))$ become less significant.  
\anold{This is different to the choice considered in \cite{GLS1} as the derivation of optimal conditions in our present contribution requires a differentiable $g$.}
What is not present in the previous work \cite{GLS1} is the coefficient $m(t)$, and when paired with $k(\vp)$ we use the product $m(t) k(\vp)$ to model a cytotoxic drug-induced decrease in tumour proliferation.  A motivating example for $m(t)$ from \cite{ColliGomez} is
\begin{align}\label{intro:m}
m(t) = \sum_{i=1}^{n} d_c e^{-\frac{t - T_i}{\tau}} H(t-T_i),
\end{align}
with drug dosage $d_c$, drug delivery times $T_i$ for $i = 1, \dots, n$, where $n$ is the number of chemotherapy cycles, $\tau$ denoting the mean lifetime of the drug and $H$ is the Heaviside function.  After the $\anold{i}$'th infusion, the effect of the drug decreases exponentially until the next infusion at time $T_{\anold{i}+1}$.  For drugs with sufficiently short mean lifetime $\tau$, or with large enough infusion gap $T_i - T_{i-1}$, \todo{there are certain time intervals where the coefficient $m$ is close to zero}. 

Similarly, the source term $S(\vp, \si)$ in \eqref{s:si} accounts for nutrient consumption and transport to and from external capillaries. The example we will use is of the form
\begin{align*}
S(\vp, \si) = - h(\vp) (\lambda_c \si - s(t))  + B(\si_c - \si) 
\end{align*}
with constant consumption rate $\lambda_c \geq 0$, capillary supply rate $B\geq 0$, capillary nutrient concentration $\si_c$ and a Lipschitz, bounded function $h$.  For instance, we can model nutrient consumption only by the tumour cells by prescribing the conditions $h(1) = 1$ and $h(-1) = 0$.  A new element absent from \cite{GLS1} is the coefficient $s(t)$, which models the reduction in nutrient supply caused by antiangiogenic therapy, and in \cite{ColliGomez} a similar form to \eqref{intro:m} is proposed for $s(t)$, meaning that under suitable conditions, the coefficient $s(t)$ \todo{take values close to zero} for certain time intervals.

It is common to prescribe cytotoxic drugs in chemotherapy that serve to disrupt the cellular division process and promote apoptosis, but tumours can overcome these effects by developing drug resistance or by generating new vasculatures through angiogenesis to obtain nutrients that compensate any loss of mass from chemotherapy.  Therefore, in certain situations, it is of interest to combine two or more different therapies so that their joint effect can account for more mechanisms that allows tumours to avoid complete elimination, and have an overall larger positive impact on the treatment than the individual monotherapies.  Unfortunately, the results of various experimental and clinical studies (see \cite{Ma} and the references cited therein) have not produced clear guidelines on how to proceed with combined therapies, in part due to the multitude of drugs presently available and patient-specific interactions of multiple drugs.  Hence, mathematicians and physicians have turn towards the framework of optimal control to infer protocols, dosages and timings that maximise tumour reduction and minimise harmful side-effects \cite{Jarrett,Led,Les,Mellal,Panetta}.  To contribute to this effort, we study an optimal control problem with the model \eqref{state} as the state system, and as controls we work with the boundary nutrient supply $w_1 = \si_B$, the cytotoxic coefficient $w_2 = m(t)$ and the antiangiogentic coefficient $w_3 = s(t)$.  The cost functional we consider is
\begin{align}\label{opt}
\notag & J(\vp, \u, w_1, w_2, w_3) \\
\notag &\quad := \frac{\alpha_\Omega}{2} \| \vp(T) - \vp_\Omega \|_{L^2(\Omega)}^2 \anold{+ \frac{\alpha_Q}{2} \|\vp - \vp_Q \|^2_{L^2(Q)} +  \frac{\alpha_{\E}}{2}\int_{Q}}  n(x,\vp) |\W_{,\E}(\vp, \E(\u))|^2\dx \dt \\
\notag & \qquad + \frac{\gamma_1}{2} \| w_1 \|_{L^2(\Sigma)}^2 + \frac{\gamma_2}{2} \| w_2 \|_{L^2(0,T)}^2 + \frac{\gamma_3}{2} \| w_3 \|_{L^2(0,T)}^2 \\
& \qquad + \gamma_4 \| w_2 \|_{L^1(0,T)} + \gamma_5 \| w_3 \|_{L^1(0,T)}.
\end{align}
It is composed of the standard tracking-type with weights $\alpha_Q$, $\alpha_\Omega \geq 0$ and target functions $\vp_Q : Q \to \R$ and $\vp_\Omega : \Omega \to \R$, and $L^2$-regularisations for the optimal controls $w_1 = \si_B$, $w_2 = m(t)$ and $w_3 = s(t)$ with corresponding weights $\gamma_1, \gamma_2,\gamma_3 \geq 0$.  
\anold{Let us stress that the controls $w_2$ and $w_3$ are solely functions of time and are spatially constant.}
Compared to previous works on the optimal control with phase field tumour models, we have the presence of a term involving the the square of the stress $\W_{,\E}(\vp, \E(\u))$ weighted by a non-negative coefficient $n(x,\vp)$ and constant $\alpha_\E \geq 0$.  Due to the role of mechanical stresses on enhancing tumour growth, we are motivated to minimise stress accumulating in a certain region of the domain, such as important organs (by taking $n(x,\vp) = \chi_{D}(x)$ for a subregion $D \subset \Omega$ where $\chi_D$ is the characteristic function of the set $D$), or in certain parts of the tumour microenvironment whose location can be encoded with the help of the phase field variable $\vp$.  One example is a function $n(x,\vp) = \max(0,\min(1,\frac{1}{2}(1-\vp)))$, so that $n$ is non-zero in the non-tumour region $\{\vp = -1 \}$ and is zero in the tumour region $\{\vp = 1\}$.

Moreover, we prescribe $L^1$-regularisations of the drug concentrations $w_2$ and $w_3$ with weights $\gamma_4, \gamma_5 \geq 0$ to the cost functional \eqref{opt}, with the aim of using the combination of both $L^2$ and $L^1$\anold{-}regularisations to show sparsity, see Theorem \ref{thm:spas} below for the precise formulation.  A first work on sparse controls with phase field tumour models is \cite{ST}, where directional sparsity \cite{Herzog} of the controls, i.e., sparsity w.r.t.~space or w.r.t.~time, is shown.  Our reasoning for such considerations is in part motivated by the common practice that chemotherapies should be administrated to the patient only in very short periods of time to avoid adverse side-effects.  In the simulations performed in \cite{ColliOpt}, where an optimal control problem of a similar nature is studied with only $L^2$-regularisation terms in \anold{the} cost functional, the optimal cytotoxic drug concentration is positive over the treatment period.  In practical applications this translates \an{to} prolonged exposure and subsequent accumulation of the drugs in the body, potentially invoking damaging side-effects and may even entail a premature abortion of the medical treatment.

The goal of this paper is to study the optimal control problem \eqref{opt} subjected to the state system \eqref{state}.  Building on the well-posedness results established in \cite{GLS1}\an{,} we prove the existence of a minimiser and derive \anold{first-order} optimality conditions.   Our main result is sparsity of the optimal drug concentrations, brought about by the convex non-differentiable $L^1$-terms in \eqref{opt}.  Compared to \cite{ST}, our analysis includes the elasticity interactions in \eqref{s:mu} and in \eqref{opt}, covering both cases of $\beta > 0$ and $\beta = 0$ in \eqref{s:si} in a uniform manner, as well as different sparsity conditions for non-negative drug concentrations $m(t)$ and $s(t)$. 

We comment that tracking terms involving the nutrient concentration $\si$, such as $\| \si - \si_Q \|_{L^2(Q)}^2$ or $\| \si(T) - \si_\Omega \|_{L^2(\Omega)}^2$ if $\beta > 0$, can also be inserted into the cost functional.  Other terms of interest include the total tumour volume at time $T$ given by the spatial integral of $\frac{1}{2}(1+\vp(T))$, and thanks to the well-posedness result for \eqref{state} (see Theorem \ref{thm:state} below) we can consider other parameters as control variables, for instance the capillary nutrient concentration $\si_c$, the boundary load $\g$, the initial data $\vp_0$, $\si_0$ and \anold{the coefficients $\chi,\lambda_p,\lambda_a,\lambda_c$ in \eqref{state} in the context of parameter estimation \cite{FLS,KL}, and} even the magnitude of the treatment time $T$ \cite{Cava,GLR,SigTime}.  One can also consider spatially varying drug concentrations $m(t,x)$ and $s(t,x)$ as in \cite{ColliOpt,ST}, and the corresponding analysis to adapt to these elements would only require minor and straightforward modifications. 


The rest of the paper is organised as follows: We recall previous results in Section \ref{sec:pre}, and the existence of a minimiser to the optimal control problem is shown in Section \ref{sec:min}.  Section \ref{sec:opt} is devoted to the derivation of \anold{first-order} optimality conditions, and in Section \ref{sec:spar} we discuss the sparsity of controls.

\section{Mathematical setting and previous results}\label{sec:pre}
\subsection{Notation and useful preliminaries}
The standard Lebesgue and Sobolev spaces over $\Omega$ are denoted by $L^p := L^p(\Omega)$ and $W^{k,p} := W^{k,p}(\Omega)$ for any $p \in [1,\infty]$ and $k >0$, with corresponding norms $\no{\cdot}_{L^p}$ and $\no{\cdot}_{W^{k,p}}$.  In the case $p = 2$, these become Hilbert spaces and we use the notation $H^k := H^k(\Omega) = W^{k,2}(\Omega)$ and the norm $\no{\cdot}_{H^k}$.  For any Banach space $Z$, we denote its dual by $Z'$, and the corresponding duality pairing by $\inn{\cdot}{\cdot}_Z$.  When $Z = H^1(\Omega)$, we use the notation $\inn{\cdot}{\cdot} = \inn{\cdot}{\cdot}_{H^1}$.  The $L^2(\Omega)$-inner product is denoted by $(\cdot,\cdot)$, while the $L^2(\Gamma)$ and $L^2(\GN)$-inner products are denoted by $(\cdot,\cdot)_{\Gamma}$ and $(\cdot,\cdot)_{\GN}$, respectively.  
We define the Sobolev space
$H^2_{\bm{n}}(\Omega)$ as the set $\{f \in H^2(\Omega) \, : \, \pdnu f = 0 \text{ on } \Gamma\}$, and for the displacement $\u$, we introduce the following function
space:
\begin{align*}
X (\Omega) & := \{ \mathbf{f} \in H^1(\Omega)^d \, : \, \mathbf{f} {\vrule height 5pt depth 4pt\,}_{\GD} = \0 \},
\end{align*}
where by \cite[Thm.~6.15-4, pp.~409--410]{C}, a Korn-type inequality is valid in $X(\Omega)$, i.e., there exists a constant $C_K > 0$ such that 
\begin{align}\label{Korn}
\no{\u}_{H^1} \leq C_K \no{\E(\u)}_{L^2} \quad \forall\, \u \in X(\Omega).
\end{align}

\subsection{Assumptions and previous results}
In this work we make the following assumptions regarding parameters and functions in the model:
\begin{enumerate}[label=$(\mathrm{A \arabic*})$, ref = $\mathrm{A \arabic*}$]
\item \label{ass:beta} Let $\g\in L^2(\GN)^d$ and $\si_B \in L^\infty(\Sigma)$ be given, while $\beta, B, \kappa, \chi, \lambda_a, \lambda_p, \lambda_c, \si_c$ are non-negative constants such that at least one of $\{B, \kappa\}$ is non-zero if $\beta = 0$.  Moreover, $\bE$ and $\E^*$ are constant symmetric second order tensors while $\C$ is a constant symmetric, positive definite fourth order tensor satisfying 
\begin{align*}
 \E : \C\E \geq  c_0 |\E|^2
\end{align*}
for all symmetric second order tensors $\E \in \R^{d\times d}_{\rm sym}$ with a positive constant $c_0$.
\item \label{ass:pot} The potential $\Psi = \Psi_1 + \Psi_2$ is a non-negative function, $\Psi_i \in C^3(\R)$ for $i=1,2$, with a convex non-negative function $\Psi_1$ such that for all $r, z \in \R$,
\begin{align*}
|\Psi_2''(r)| \leq C, & \quad |\Psi_1'''(r)| \leq C( 1 + |r|), \\
|\Psi'(r) - \Psi'(z)| & \leq C \big ( 1 + |r|^2+|z|^2 \big ) |r-z|, \\
|\Psi''(r) - \Psi''(z) | & \leq C \big ( 1 + |r| + |z| \big ) |r-z|,
\end{align*}
for some positive constant $C$.
\item \label{ass:f} The functions $f$, $g$, $h$ and $k$ satisfy $f,h,k \in W^{1,\infty}(\R)$, $g \in W^{1,\infty}(\R^{d \times d}, \R)$, with Lipschitz constants that shall be denoted by a common symbol $L > 0$. Furthermore, we assume $h$ is non-negative. 
\item \label{ass:ms} The cytotoxic and antiangiogenic functions satisfy $m, s \in L^\infty(0,T)$.
\item \label{ass:ini} The initial conditions satisfy $\vp_0 \in H^1(\Omega)$ and $\si_0 \in L^2(\Omega)$ with $0 \leq \si_0 \leq M := \max (\si_c, \| \si_B \|_{L^\infty(\Sigma)})$ a.e.~in $\Omega$.
\end{enumerate}
To study the optimal control problem we will need the following:
\begin{enumerate}[resume*]
\item \label{ass:opt} We assume $f$, $h$, $k \in \an{C^2(\R)\cap}W^{2,\infty}(\R)$, $g \in \an{C^2(\R^{d \times d}, \R)\cap}W^{2,\infty}(\R^{d \times d}, \R)$ and 
\linebreak $n: \Omega \times \R \to \R$ is a Carath\'eodory function such that $n(x,\cdot) \in C^1(\R) \cap W^{1,\infty}(\R)$ is non-negative for a.e.~$x \in \Omega$.
\item \label{ass:gamma} 
The coefficients $\alpha_Q, \alpha_\Omega, \alpha_\E, \gamma_1, \gamma_2, \gamma_3, \gamma_4, \gamma_5$ are non-negative constants, not all zero.
Moreover, $\gamma_2$ is positive if $\gamma_4$ is positive and $\gamma_3$ is positive when $\gamma_5$ is positive.
\item \label{ass:target}
The objective data $\vp_Q : Q \to \R$, $\vp_\Omega : \Omega \to \R$ are given functions satisfying $\vp_Q \in L^2(Q), \vp_\Omega \in L^2(\Omega).$
\end{enumerate}
\an{It is worth noting that the conditions expressed in \eqref{ass:pot} are fulfilled by the classical quartic potential $\Psi(r)= \tfrac 14 (r^2-1)^2$.}
For the motivating example \eqref{eg:g}, for any $\bm{A} \in \R^{d \times d}$, we use the notation $g'(\bm{A})$ to denote the tensor derivative of $g$, i.e., $g'(\bm{A})$ is a second order tensor with 
\begin{align*}
[g'(\bm{A})]_{ij} = \frac{\pd}{\pd \bm{A}_{ij}} g(\bm{A}) = - \frac{\bm{A}_{ij}}{(1+|\bm{A}|^2)^{3/2}} \quad \text{ for } 1 \leq i,j \leq d.
\end{align*}
On the other hand, we use the notation $g''(\bm{A})$ to denote the Hessian of $g$, which is a fourth order tensor defined as
\begin{align*}
[g''(\bm{A})]_{ijkl} =  \frac{\pd^2}{\pd \bm{A}_{ij} \pd \bm{A}_{kl}} g(\bm{A}) = \frac{3 \bm{A}_{ij} \bm{A}_{kl}}{(1+|\bm{A}|^2)^{5/2}} - \frac{\delta_{ik} \delta_{jl}}{(1+|\bm{A}|^2)^{3/2}} \quad \text{ for } 1 \leq i,j,k,l \leq d.
\end{align*}
Hence, it is easy to see that for any $\bm{A} \in \R^{d \times d}$, both $|[g'(\bm{A})]_{ij}|$ and $|[g''(\bm{A})]_{ijkl}|$ are bounded for all $1 \leq i,j,k,l \leq d$.  In particular, we can infer that $|g'(\W_{,\E}(\vp,\E(\u)))|$ and $|g''(\W_{,\E}(\vp,\E(\u))|$ are bounded a.e.~in $Q$ thanks to \eqref{ass:f}.  For the rest of the paper, the parameters $\beta$, $\g$, $B$, $\kappa$, $\chi$, $\lambda_a$, $\lambda_p$, $\lambda_c$, $\si_c$, $\C$, $\bE$, $\E^*$, as well as initial data $\vp_0$ and $\si_0$ are kept fixed.  We then introduce the notation
\begin{align*}
\bm{w} = (w_1, w_2, w_3),
\end{align*}
and the set of admissible controls $\U_{ad} = \U_{ad}^{(1)} \times \U_{ad}^{(2)} \times \U_{ad}^{(3)}$ as
\begin{equation}\label{Uad}
\begin{aligned}
\U_{ad}^{(1)} & :=  \{ w_1 \in L^\infty(\Sigma) \, : \, \underline{w}_1 \leq w_1 \leq \overline{w}_1 \text{ a.e.~on } \Sigma \}, \\
\U_{ad}^{(i)} & := \{ w_i \in L^\infty(0,T) \, : \, \underline{w}_i \leq w_i \leq \overline{w}_i \text{ a.e.~in } (0,T) \} \text{ for } i = 2,3,
\end{aligned}
\end{equation}
with fixed $\underline{w}_1$, $\overline{w}_1 \in L^\infty(\Sigma)$,
$\underline{w}_2$, $\overline{w}_2$, $\underline{w}_3$, $\overline{w}_3 \in L^\infty(0,T)$
such that
\an{$\underline{w}_1 \leq \overline{w}_1$ a.e.~on $\Sigma$, $\underline{w}_i \leq \overline{w}_i$ a.e.~in  $(0,T)$ for $i=2,3$,
and $\max(\| \underline{w}_1 \|_{L^\infty(\Sigma)},\| \overline{w}_1 \|_{L^\infty(\Sigma)}) \leq M$.}
The admissible set of controls $\U_{ad}$ is a non-empty, closed and convex subset of $\U := L^2(\Sigma) \times L^2(0,T) \times L^2(0,T)$, and we can find a positive constant $R$ such that 
\begin{align*}
\U_R := \{ (w_1, w_2, w_3) \in \U \, : \, \| w_1 \|_{L^2(\Sigma)} + \| w_2 \|_{L^2(0,T)} + \| w_3 \|_{L^2(0,T)} < R \} \supset \U_{ad}.
\end{align*}

\noindent The following result concern\anold{s the} well-posedness of the model \eqref{state}.
\begin{thm}\label{thm:state}
Under \eqref{ass:beta}-\eqref{ass:ini} there exists a unique weak solution $\anold{(\vp, \mu, \si, \u)}$ to \eqref{state} and an exponent $p > 2$ such that
\begin{align*}
\vp & \in H^1(0,T;H^1(\Omega)') \cap L^\infty(0,T;H^1(\Omega)) \cap L^2(0,T;H^2_{\bm{n}}(\Omega)), \\
\mu & \in L^2(0,T;H^1(\Omega)), \\
\si & \in L^2(0,T;H^1(\Omega)) \cap L^\infty(0,T;L^\infty(\Omega)) \text{ with } 0 \leq \si \leq M \text{ a.e.~in } Q, \\
& \text{ and } \si \in H^1(0,T;H^1(\Omega)') \cap L^\infty(0,T;L^2(\Omega)) \text{ if } \beta > 0, \\
\anold{\u} & \in \anold{L^\infty(0,T; \X \cap W^{1,p}(\Omega)),}
\end{align*}
with $\vp(0) = \vp_0$ in $L^2(\Omega)$ as well as $\si(0) = \si_0$ in $L^2(\Omega)$ if $\beta > 0$, and 
\begin{subequations}
\begin{alignat}{2}
0 & = \int_0^T \inn{\vp_t}{\zeta} + (\nx \mu, \nx \zeta) - (U(\vp, \si, \E(\u)), \zeta) \dt, \label{w:1} \\
 0 &= \int_0^T (\mu, \zeta) - (\nx \vp, \nx \zeta) - (\an{\Psi}'(\vp), \zeta) + \chi (\si, \zeta) - (\W_{,\vp}(\vp, \E(\u)), \zeta) \dt, \label{w:2} \\
0 & = \int_0^T \beta \inn{\si_t}{\zeta} + (\nx \si, \nx \zeta) + \kappa(\si - \si_B,\zeta)_{\Gamma} - (S(\vp,\si), \zeta) \dt, \label{w:3} \\
0 & = \int_0^T (\C(\E(\u) - \bE - \vp \E^*), \nx \bet) -( \g , \bet )_{\GN} \dt 
\label{w:4}
\end{alignat}
\end{subequations}
for all $\zeta \in L^2(0,T;H^1(\Omega))$ and $\bet \in L^2(0,T; \X)$.  Moreover, there exists a positive constant $K_1$ independent of $\beta$ such that 
\begin{equation}\label{state:bdd}
\begin{aligned}
& \| \vp \|_{H^1(0,T;H^1(\Omega)') \cap L^\infty(0,T;H^1) \cap L^2(0,T;H^2)} + \| \mu \|_{L^2(0,T;H^1)} \\
& \quad + \| \si \|_{L^2(0,T;H^1)} + \beta^{\frac{1}{2}} \| \si \|_{H^1(0,T;H^1(\Omega)') \cap L^\infty(0,T;L^2)} \\
& \quad + \| \u \|_{L^\infty(0,T;\X \cap W^{1,p}(\Omega))} \leq K_1 \big ( 1 + \beta^{\frac{1}{2}} \| \si_0 \|_{L^2} \big ).
\end{aligned}
\end{equation}
For any pair $\{\anold{(\vp_i, \mu_i, \si_i, \u_i)}\}_{i=1,2}$ of weak solutions to \eqref{state} corresponding to data 
\begin{align*}
\anold{\{(\vp_{0,i}, \si_{0,i}, \g_{i}, \si_{B,i}, m_{i}, s_i)\}_{i=1,2},}
\end{align*}
there exists a positive constant $K_2$ independent of the differences of $\{\anold{(\vp_i, \mu_i, \si_i, \u_i)}\}_{i=1,2}$  and $\beta$ such that
\begin{equation}\label{state:ctsdep}
\begin{aligned}
& \| \vp_1 - \vp_2 \|_{L^\infty(0,T;H^1) \cap L^2(0,T;H^2)} + \| \mu_1 - \mu_2 \|_{L^2(0,T;H^1)} + \| \si_1 - \si_2 \|_{L^2(0,T;H^1)} \\
& \qquad + \beta^{\frac{1}{2}} \| \si_1 - \si_2 \|_{L^\infty(0,T;L^2)} + \| \u_1 - \u_2 \|_{L^\infty(0,T;\X)} \\
& \quad \leq K_2 \Big ( \| \vp_{0,1} - \vp_{0,2} \|_{H^1} + \beta^{\frac{1}{2}} \| \si_{0,1} - \si_{0,2} \|_{L^2} + \| \g_1 - \g_2 \|_{L^2(\GN)} \Big ) \\
& \qquad + K_2 \Big ( \| \si_{B,1} - \si_{B,2} \|_{L^2(\Sigma)}
 \anold{+ \| m_1 - m_2 \|_{L^2(0,T)} + \| s_1 - s_2 \|_{L^2(0,T)} \Big ).}
\end{aligned}
\end{equation}
\end{thm}

\begin{remark}
The proof of existence can be deduced analogously from \cite[Sec.~3]{GLS1}, and we comment that the subsequent constant $K_1$ in \eqref{state:bdd} is bounded uniformly in $(\si_B, m(t), s(t))$ when we restrict to the \an{open} set $\U_R$.  Whereas a minor modification of \cite[Sec.~6]{GLS1} using the boundedness of $k$ and $h$ yields the above continuous dependence assertion in the presence of the new coefficients $m(t)$ and $s(t)$.  Hence, we omit the details.
\end{remark}

\begin{remark}
A closer inspection of the proof in \cite[Sec.~5.2]{GLS1} allows us to deduce the \anold{further} regularity statement
\begin{align*}
\vp \in L^4(0,T;H^2_{\bm{n}}(\Omega)).
\end{align*}
We briefly sketch the argument.  Testing \eqref{s:mu} with $-\Lx \vp$, integrating by parts for the terms involving $\mu$ and $\Psi'(\vp)$, then using the convexity of $\Psi_1$, the bounds for $\Psi_2''$, the boundedness of $\si$, and the regularity $\vp \in L^\infty(0,T;H^1(\Omega))$ and $\u \in L^\infty(0,T;\X)$,
\begin{align*}
\tfrac{1}{2} \| \Lx \vp \|_{L^2}^2 & \leq \| \nabla \mu \|_{L^2} \| \nabla \vp \|_{L^2} + C \anold{ \|\nabla \vp \|_{L^2}^2} + C \| \si \|_{L^2}^2 + C \| \W_{,\vp}(\vp, \E(\u)) \|_{L^2}^2 \\
& \leq C \big ( 1 + \| \nabla \mu \|_{L^2} \big ).
\end{align*}
Squaring and integrating over $(0,T)$ yields that $\Lx \vp \in L^4(0,T;L^2(\Omega))$ and elliptic regularity gives the assertion.
\end{remark}

\section{The optimal control problem}\label{sec:min}
In this section we show \anold{that} there exists at least one minimiser to the optimal control problem \anold{minimising the cost functional \eqref{opt} with state system given by \eqref{state}}.  By Theorem \ref{thm:state}, we can define the control-to-state operator $\bS$ which assigns every admissible control $\bm{w} = (w_1, w_2, w_3) = (\si_B, m, s)$ the corresponding unique solution $(\vp, \mu, \si, \u)$ to \eqref{state}.  Namely, we have
\begin{align*}
\bS : \U_{ad} \subset \U_R \to \bY, \quad (w_1, w_2, w_3) \mapsto (\vp, \mu, \si, \u),
\end{align*}
where the solution space $\bY$ is defined\an{, according to Theorem \ref{thm:state},} as
\begin{align*}
\bY := \begin{cases}
H^1(0,T;H^1(\Omega)') \cap L^\infty(0,T;H^1(\Omega)) \cap L^2(0,T;H^2_{\bm{n}}(\Omega)) \times L^2(0,T;H^1(\Omega)) &  \\
\times H^1(0,T;H^1(\Omega)') \cap L^\infty(0,T;L^2(\Omega)) \cap L^2(0,T;H^1(\Omega)) \cap L^\infty(Q) & \\
\times L^\infty(0,T;\X \cap W^{1,p}(\Omega)) \quad \text{ if } \beta > 0, \\
\\
H^1(0,T;H^1(\Omega)') \cap L^\infty(0,T;H^1(\Omega)) \cap L^2(0,T;H^2_{\bm{n}}(\Omega)) \times L^2(0,T;H^1(\Omega)) &  \\
\times L^2(0,T;H^1(\Omega)) \cap L^\infty(Q) & \\
\times L^\infty(0,T;\X \cap W^{1,p}(\Omega)) \quad \text{ if } \beta = 0\anold{.}
\end{cases}
\end{align*}
Denoting by $\bS_1(\bm{w}) = \vp$ the first component and by $\bS_4(\bm{w}) = \u$ the fourth component, we can define the reduced cost functional as
\begin{align*}
\J(\bm{w}) = J(\bS_1(\bm{w}), \bS_4(\bm{w}), \bm{w}).
\end{align*}
\begin{thm}
Under \eqref{ass:beta}-\eqref{ass:target}, there exists at least one minimiser $\bm{w} = (w_1, w_2, w_3) \in \U_{ad}$ to the optimal control problem
\begin{align*}
\min_{(z_1, z_2, z_3) \in \U_{ad}} \J(z_1, z_2, z_3).
\end{align*}
\end{thm}
Since the proof is somewhat standard we omit the details and sketch the main points.  The non-negativity of $\J$ implies the infimum $\inf_{\U_{ad}} \J$ exists and allows us to find a minimising sequence $\{\bm{w}_n = (w_{1,n},w_{2,n}, w_{3,n})\}_{n \in \N} \subset \U_{ad}$ such that $\J(\bm{w}_n) \to \inf_{\U_{ad}} \J$ \anold{as $n\to\infty$}.  Denoting the corresponding solution as $\anold{(\vp_n, \mu_n, \si_n, \u_n)} = \bS(\bm{w}_n) \in \bY$, we infer by the bound \eqref{state:bdd} that $\{\anold{(\vp_n, \mu_n, \si_n, \u_n)}\}_{n \in \N}$ is uniformly bounded in $\bY$.  Hence, along a non-relabelled subsequence there exists a limit triplet $\bm{w} = (w_1, w_2, w_3) \in \U_{ad}$ such that\an{, as $n\to\infty$,}
\begin{align*}
(w_{1,n}, w_{2,n}, w_{3,n}) &\to (w_1, w_2, w_3) \quad \text{ weakly* in } L^\infty(\Sigma) \times L^\infty(0,T) \times L^\infty(0,T), \\
\anold{(\vp_n, \mu_n, \si_n, \u_n)} & \to (\vp, \mu, \si, \u) = \bS(\bm{w}) \quad  \text{ weakly* in } \bY.
\end{align*}
The Aubin--Lions compactness theorem then yields the strong convergence of $\vp_n$ to $\vp$ in $C^0([0,T];L^2(\Omega))$, allowing us to pass to the limit in the tracking terms of \anold{$\J$} and provides strong convergence $\sqrt{n(\vp_n)}\bet \to \sqrt{n(\vp)} \bet$ for all $\bet \in L^2(Q)$.  Together with the weak convergence of $\E(\u_n)$ to $\E(\u)$ in $L^2(Q)$, we arrive at the weak convergence
\begin{align*}
\sqrt{n(\vp_n)} \W_{,\E}(\vp_n, \E(\u_n)) \an{\to} \sqrt{n(\vp)} \W_{,\E}(\vp, \E(\u)) \quad \text{ \an{weakly} in } L^2(Q).
\end{align*}
Then, by the weak lower semicontinuity of $L^p$-norms for $p \in [1,\infty)$, we deduce that 
\begin{align*}
\J(\bm{w}) \leq \liminf_{n \to \infty} \J(\bm{w}_n) = \inf_{\U_{ad}} \J.
\end{align*}

\section{\anold{First-order} necessary optimality conditions}\label{sec:opt}
\begin{thm}
\label{thm:optcond}
Under \eqref{ass:beta}-\eqref{ass:target}, let $\bm{w}^* = (w_1^*, w_2^*, w_3^*) \in \U_{ad}$ be an optimal control with associated state $(\vp, \mu, \si, \u) = \bS(\bm{w}^*)$.  Then, there exist functions $\an{\lambda_2, \lambda_3} \in L^\infty(0,T)$ such that\an{,} for a.e.~$t \in (0,T)$,
\begin{align}\label{lambd}
\lambda_i(t) \in \begin{cases}
\{1\} & \text{ if } w_i^*(t) > 0, \\
[-1,1] & \text{ if } w_i^*(t) = 0, \\
\{-1\} & \text{ if } w_i^*(t) < 0,
\end{cases} \quad \text{ for } i \in \{2,3\},
\end{align}
and for all $\bm{y} = (y_1, y_2, y_3) \in \U_{ad}$,
\anold{\begin{equation}\label{nes:opt}
\begin{aligned}
0 & \leq \int_0^T   (\kappa r + \gamma_1 w_1^*, y_1 - w_1^*)_\Gamma \dt + \int_0^T  (\gamma_2 w_2^* + \gamma_4 \lambda_2 - \int_\Omega k(\vp) p \dx)( y_2 - w_2^* ) \dt \\
& \quad  + \int_0^T   (\gamma_3 w_3^* + \gamma_5 \lambda_3 + \int_\Omega h(\vp) r \dx )( y_3 - w_3^*  ) \dt,
\end{aligned}
\end{equation}}
where $p$ and $r$ are the first and third components of the associated adjoint variables $(p,q,r,\bs)$ satisfying the adjoint system \eqref{adj}.
\end{thm}
The proof of \an{Theorem \ref{thm:optcond}} proceeds in four steps, which is covered in the following four subsections.
\subsection{Linearised state system}
Given $\bm{w}^* = (w_1^*, w_2^*, w_3^*) \in \U_{ad}$ with associated \an{state} $\anold{(\vp, \mu, \si, \u)} = \bS(\bm{w}^*) \in \bY$, for arbitrary $\bh = (h_1, h_2, h_3) \in \U$, we study the following linearised state system:
\begin{subequations}\label{lin}
\begin{alignat}{3}
\xi_t & = \Lx \eta + U_{\mathrm{lin}}(\vp, \si, \E(\u), w_2^*, h_2, \xi, \psi, \E(\bv)) && \quad\text{ in } Q, \\
U_{\mathrm{lin}} & = \lambda_p g(\W_{,\E}(\vp, \E(\u))) (f'(\vp)\xi \si + f(\vp) \psi) && \\
\notag & \quad + \lambda_p \si f(\vp) g'(\W_{,\E}(\vp, \E(\u))) : \anold{\C(\E(\bv) -  \xi \E^*)} && \\
\notag & \quad  - (\lambda_a + \anold{w_2^*}) k'(\vp) \xi - h_2 k(\vp) \\
\eta & = - \Lx \xi + \Psi''(\vp) \xi - \chi \psi - \C(\E(\bv) - \xi \E^*) : \E^* && \quad\text{ in } Q, \\
\beta \psi_t & = \Lx \psi  + S_{\mathrm{lin}}(\vp, \si, w_3^*, h_3, \xi, \psi) && \quad\text{ in } Q, \\
S_{\mathrm{lin}} & = - h'(\vp) \xi (\lambda_c \si - \anold{w_3^*}) - h(\vp)(\lambda_c \psi - h_3) - B \psi && \\
\anold{\0} & = \div (\C(\E(\bv) - \xi \E^*)) && \quad\text{ in } Q, \\
0 & = \xi(0) = \psi(0) && \quad\text{ in } \Omega, \\
0 & = \pdnu \xi = \pdnu \eta, \quad \pdnu \psi + \kappa (\psi - h_1) = 0 && \quad\text{ on } \Sigma, \\
\bv & = \0 && \quad\text{ on } \Sigma_D, \\
\0 & = \C(\E(\bv) - \xi \E^*) \bn && \quad\text{ on } \Sigma_N.
\end{alignat}
\end{subequations}
Introducing the solution space
\begin{align*}
\bYl = \begin{cases}
H^1(0,T;H^1(\Omega)') \cap L^\infty(0,T;H^1(\Omega)) \cap L^2(0,T;H^2_{\bm{n}}(\Omega)) \times L^2(0,T;H^1(\Omega)) &  \\
\times H^1(0,T;H^1(\Omega)') \cap L^\infty(0,T;L^2(\Omega)) \cap L^2(0,T;H^1(\Omega)) & \\
\times L^\infty(0,T;\X) \quad \text{ if } \beta > 0, \\
\\
H^1(0,T;H^1(\Omega)') \cap L^\infty(0,T;H^1(\Omega)) \cap L^2(0,T;H^2_{\bm{n}}(\Omega)) \times L^2(0,T;H^1(\Omega)) &  \\
\times L^2(0,T;H^1(\Omega)) \times L^\infty(0,T;\X) \quad \text{ if } \beta = 0,
\end{cases}
\end{align*}
we have the following result.
\begin{thm}\label{thm:lin}
For given $\bm{w}^* = (w_1^*,w_2^*, w_3^*) \in \U_{ad}$ with $(\vp, \mu, \si, \u) = \bS(\bm{w}^*) \in \bY$ and $\bm{h} = (h_1, h_2, h_3) \in \U$, under \eqref{ass:beta}-\an{\eqref{ass:opt}}, there exists a unique solution $(\xi, \eta, \psi, \bv) \in \bYl$ to \eqref{lin} with $\xi(0) = 0$, $\psi(0) = 0$ if $\beta > 0$, and
\begin{subequations}\label{lin:weak}
\begin{alignat}{3}
0 & = \int_0^T \inn{\xi_t}{\zeta} + (\nabla \eta , \nabla \zeta) - (U_{\mathrm{lin}}(\vp, \si, \E(\u), w_2^*, h_2, \psi, \xi, \E(\bv)), \zeta) \dt, \label{lin:1} \\
0 & = \int_0^T (\eta - \Psi''(\vp) \xi + \chi \psi + \C(\E(\bv) - \xi \E^*): \E^*, \zeta) - (\nabla \xi, \nabla \zeta) \dt, \label{lin:2} \\
0 & = \int_0^T \beta \inn{\psi_t}{\zeta} + (\nabla \psi, \nabla \zeta)  + \kappa (\psi - h_1, \zeta)_\Gamma -  (S_{\mathrm{lin}}(\vp, \si, w_3^*, h_3, \xi, \psi), \zeta) \dt, \label{lin:3} \\
0 & = \int_0^T (\C(\E(\bv) - \xi \E^*), \nabla \bet) \dt, \label{lin:4}
\end{alignat}
\end{subequations}
for all $\zeta \in L^2(0,T;H^1(\Omega))$ and $\bet \in L^2(0,T;\X)$.
\end{thm}
\begin{proof}
We proceed with formal estimates that can be justified rigorously with a Galerkin approximation.  In the following the positive constants denoted by the symbol $C$ will be independent of the Galerkin parameter, as well as $h_1$, $h_2$ and $h_3$ \anold{and might change from line to line}. 
\anold{Besides, let us remark that since $\bm{w}^*$ is fixed, the corresponding state $(\vp, \mu, \si, \u)\an{={\cal S}(\bm{w}^*)}$ enjoys the bound \eqref{state:bdd}.} Let us mention that uniqueness follows from existence thanks to the linearity of the system \eqref{lin}. 

We test \eqref{lin:1} with $\eta$ and $K \xi$, \eqref{lin:2} with $-\xi_t$ and $\eta$, \eqref{lin:3} with $R \psi$ and \eqref{lin:4} with $\bv_t$ for positive constants $K, R$ to be determined later.  After summing and rearranging we get
\begin{equation}\label{lineq}
\begin{aligned}
& \frac{1}{2} \frac{d}{dt} \Big ( K \| \xi \|_{L^2}^2 + \| \nabla \xi \|_{L^2}^2 + R \beta \| \psi \|_{L^2}^2 \Big ) + \frac{d}{dt} \int_\Omega \W^{\mathrm{lin}}(\xi, \E(\bv)) \dx \\
& \qquad \an{+ \| \eta \|_{H^1}^2 }
+ R \| \nabla \psi \|_{L^2}^2 + R \kappa \| \psi \|_{L^2_\Gamma}^2 + RB \| \psi \|_{L^2}^2 \\
& \quad = - K (\nabla \eta, \nabla \xi) + (U_{\mathrm{lin}}, \eta + K \xi) + (\Psi''(\vp) \xi, \eta - \xi_t) + \chi (\psi, \xi_t - \eta) \\
& \qquad + (\nabla \xi, \nabla \eta) - (\C(\E(\bv) - \xi \E^*) : \E^*, \eta) + R \kappa (h_1, \psi)_{\Gamma} + R(S_{\mathrm{lin}}, \psi) 
\end{aligned}
\end{equation}
where
\begin{align*}
\W^{\mathrm{lin}}(\xi, \E(\bv)) = \frac{1}{2}(\E(\bv) - \xi \E^*) : \C (\E(\bv) - \xi \E^*),
\end{align*}
and we have used the identity
\begin{align*}
\frac{d}{dt} \int_\Omega \W^{\mathrm{lin}}(\xi, \E(\bv)) \dx & = (\W^{\mathrm{lin}}_{,\xi}(\xi, \E(\bv)), \xi_t) + (\W^{\mathrm{lin}}_{,\E}(\xi, \E(\bv)), \E(\bv_t)) \\
& = -(\C(\E(\bv) - \xi \E^*): \E^*, \xi_t) + (\C(\E(\bv) - \xi \E^*), \E(\bv_t))
\end{align*}
derived with the help of the symmetry of $\C$ and $\E^*$.  Furthermore, by the positive definiteness of $\C$ and Young's inequality,
\begin{align}\label{wlin}
\W^{\mathrm{lin}}(\xi, \E(\bv)) \geq \frac{c_0}{4} \anold{\| \E(\bv) \|^2_{L^2}} - C \big ( 1 + \| \xi \|_{L^2}^2 \big ).
\end{align}
Next, \anold{recalling that $w_2^*$ and $w_3^*$ are constant in space,} from the definition of $U_{\mathrm{lin}}$ and $S_{\mathrm{lin}}$, using the boundedness of $w_2^*$, $w_3^*$, $f$, $g$, $h$, $k$ and their derivatives, as well as the boundedness of $\si$, and of $|g'(\W_{,\E}(\vp, \E(\u)))|$, it is easy to see that
\begin{align}
	\label{Ulin}
\| U_{\mathrm{lin}} \|_{L^2} & \leq C \big ( \| \xi \|_{L^2} + \| \psi \|_{L^2} + \| \E(\bv) \|_{L^2} + \anold{|h_2|} \big ),
\end{align}
while
\begin{align*}
R( S_{\mathrm{lin}}, \psi) & \leq - R\big ( B \| \psi \|_{L^2}^2 + \lambda_c (h(\vp) \psi, \psi) \big ) + R C \big ( \| \xi \|_{L^2} + \anold{|h_3|} \big ) \| \psi \|_{L^2} \\
& \leq C \big ( \anold{\| \xi \|_{L^2}^2 }+ \anold{|h_3|}^2 \big ) + \| \psi \|_{L^2}^2\anold{,}
\end{align*}
\anold{where we also use that the first two terms on the right-hand side are non-positive.}
Also, using \eqref{ass:pot}\an{, the continuous inclusion $V \subset L^6(\Omega)$,} and the fact that $\vp \in L^\infty(0,T;H^1(\Omega))$, 
\begin{align*}
\| \Psi''(\vp) \xi \|_{L^2}^2 \leq C \| \xi \|_{L^6}^2 \big ( 1 + \| \vp \|_{L^6}^4 \big ) \leq C \| \xi \|_{H^1}^2.
\end{align*}
Hence, the right-hand side of \eqref{lineq} can be estimated as follows:
\begin{align*}
\mathrm{RHS} & \leq \frac{1}{4} \| \eta \|_{H^1}^2 + c \| \psi \|_{L^2}^2 + C \big ( \| \xi \|_{H^1}^2 + \| \E(\bv) \|_{L^2}^2 + \| h_1 \|_{L^2_\Gamma}^2 + \anold{|h_2|}^2 + \anold{|h_3|}^2  \big ) \\
& \quad  + (\chi \psi - \Psi''(\vp) \xi, \xi_t) + \frac{R\kappa}{2} \| \psi \|_{L^2_\Gamma}^2 - RB \| \psi \|_{L^2}^2
\end{align*}
with a positive constant $c$ that is independent of $R$.  To handle the term involving $\xi_t$, we use \eqref{lin:1} \an{and \eqref{Ulin}} to deduce that 
\begin{equation}\label{lin:xit}
\begin{aligned}
\| \xi_t \|_{H^1(\Omega)'} & \leq \| \nabla \eta \|_{L^2} + \| U_{\mathrm{lin}} \|_{L^2} \\
& \leq \| \nabla \eta \|_{L^2} + C \big ( \| \xi \|_{L^2} + \| \psi \|_{L^2} + \| \E(\bv) \|_{L^2} + \anold{|h_2|} \big ),
\end{aligned}
\end{equation}
while invoking the assumption \eqref{ass:pot} for $\Psi''$ and $\Psi'''$ leads to
\begin{align*}
\| \Psi''(\vp) \xi \|_{H^1} & \leq \| \Psi''(\vp) \xi \|_{L^2} + \| \xi \Psi'''(\vp) \nabla \vp \|_{L^2} + \| \Psi''(\vp) \nabla \xi \|_{L^2} \\
& \leq C \| \xi \|_{H^1} + \| \Psi'''(\vp)\|_{L^6} \| \xi \|_{L^6} \| \nabla \vp \|_{L^6} + \| \Psi''(\vp) \|_{L^\infty} \| \nabla \xi \|_{L^2} \\
& \leq C \big ( 1 + \| \vp \|_{H^2}^2 \big ) \| \xi \|_{H^1}.
\end{align*}
Then, via Young's inequality
\begin{align*}
|(\Psi''(\vp) \xi, \xi_t)| \leq \frac{1}{8} \| \nabla \eta \|_{L^2}^2 + C \big ( 1 + \| \vp \|_{H^2}^4 \big ) \| \xi \|_{H^1}^2 + C \big ( \| \psi \|_{L^2}^2 + \| \E(\bv) \|_{L^2}^2 + \anold{|h_2|}^2 \big ).
\end{align*}
Meanwhile, by a similar argument,
\begin{align*}
|(\chi \psi, \xi_t)| \leq \frac{1}{8} \| \nabla \eta \|_{L^2}^2 + C \big ( \| \psi \|_{H^1}^2 + \| \E(\bv) \|_{L^2}^2 + \anold{|h_2|}^2 + \| \xi \|_{L^2}^2 \big ),
\end{align*}
and so collecting the above estimates for the right-hand side of \eqref{lineq} we deduce the existence of two positive constants \anold{$c_1$ and $c_2$} independent of $R$ such that
\begin{align*}
& \frac{1}{2} \frac{d}{dt} \Big ( K \| \xi \|_{L^2}^2 + \| \nabla \xi \|_{L^2}^2 + R \beta \| \psi \|_{L^2}^2 \Big ) + \frac{d}{dt} \int_\Omega \W^{\mathrm{lin}}(\xi, \E(\bv)) \dx \\
& \qquad + \frac{1}{2} \| \eta \|_{H^1}^2 + (R - \anold{c_1}) \| \nabla \psi \|_{L^2}^2 + \frac{R\kappa}{2} \| \psi \|_{L^2_\Gamma}^2 \anold{+ (RB- c_2 ) \| \psi \|_{L^2}^2} \\
& \quad \leq C \big ( 1 + \| \vp \|_{H^2}^4 \big ) \big ( \| \E(\bv) \|_{L^2} + \|\xi \|_{H^1}^2 \big ) + C \big ( \| h_1 \|_{L^2_\Gamma}^2 + \anold{|h_2|}^2 + \anold{|h_3|}^2 \big ) \\
& \quad \leq C \Big ( 1 + \| \vp \|_{H^2}^4 \big ) \Big ( \| \W^{\mathrm{lin}}(\xi, \E(\bv)) \|_{L^1} + \| \xi \|_{H^1}^2 \big ) + C \big ( \| h_1 \|_{L^2_\Gamma}^2 + \anold{|h_2|}^2 + \anold{|h_3|}^2 \big ),
\end{align*}
where we have also used the lower bound \eqref{wlin}. 
\anold{In the case $\beta >0$ we can directly employ Gronwall's inequality to handle the terms on the right-hand side, whereas in the case $\beta = 0$ we proceed as follows: if}
$B > 0$, we can choose $R> \max(2 \anold{c_1}, \frac{2\anold{c_2}}{B})$, and if $B = 0$ then $\kappa > 0$ by \eqref{ass:beta} and we employ the generalised Poincar\'e inequality
\begin{align*}
\| f \|_{L^2} \leq C \big ( \| \nabla f \|_{L^2} + \| f \|_{L^2_\Gamma} \big ) \quad \forall f \in H^1(\Omega)
\end{align*}
to handle the term $\anold{c_2} \| \psi \|_{L^2}^2$ on the left-hand side after choosing $R$ sufficiently large.  Invoking Gronwall's inequality, keeping in mind that $\vp \in L^4(0,T;H^2(\Omega))$, there exists a constant $C$ independent of $\beta$ such that 
\begin{equation}\label{lineq1}
\begin{aligned}
& \| \xi \|_{L^\infty(0,T;H^1)}^2 + \beta \| \psi \|_{L^\infty(0,T;L^2)}^2 + \| \W^{\mathrm{lin}}(\xi, \E(\bv)) \|_{L^\infty(0,T;L^1)}^2 \\
& \qquad + \| \eta \|_{L^2(0,T;H^1)}^2 + \| \nabla \psi \|_{L^2(Q)}^2 + \kappa \| \psi \|_{L^2(\Sigma)}^2 + B \| \psi \|_{L^2(Q)}^2 \\
& \quad \leq C \big ( \| h_1 \|_{L^2(\Sigma)}^2 +\anold{ \| h_2 \|_{L^2(0,T)}^2 + \| h_3 \|_{L^2(0,T)}^2} \big ).
\end{aligned}
\end{equation}
In view of $\xi(0) = 0$, we note that from \eqref{lin} the initial data $\bv_0$ assigned to $\bv$ satisfies the elliptic equation
\begin{align*}
\begin{cases}
\div (\C(\E(\bv_0))) = \an{\0} & \text{ in } \Omega, \\
\bv_0 = \0 & \text{ on } \GD, \\
\C(\E(\bv_0))\bm{n} = \0 & \text{ on } \GN\anold{.}
\end{cases}
\end{align*}
Testing with $\bv_0$ and using Korn's inequality shows that 
\begin{align*}
\| \bv_0 \|_{H^1} \leq \anold{C_K} \| \E(\bv_0) \|_{L^2} \anold{\leq} \anold{\frac {C_K}{c_0}} (\C \E(\bv_0), \E(\bv_0)) = 0,
\end{align*}
which explains the absence of initial data on the right-hand side of \eqref{lineq1}.  Then, recalling the lower bound \eqref{wlin} and employing Korn's inequality we have
\begin{align*}
& \| \xi \|_{L^\infty(0,T;H^1)}^2 + \beta \| \psi \|_{L^\infty(0,T;L^2)}^2 + \| \bv \|_{L^\infty(0,T;H^1)}^2  + \| \eta \|_{L^2(0,T;H^1)}^2 + \| \psi \|_{L^2(0,T;H^1)}^2  \\
& \quad \leq C \big ( \| h_1 \|_{L^2(\Sigma)}^2 + \| h_2 \|_{L^2\anold{(0,T)}}^2 + \| h_3 \|_{L^2\anold{(0,T)}}^2 \big ),
\end{align*}
which also implies the uniqueness of solution since the difference of two solutions to the linear system \eqref{lin} satisfies \eqref{lin} with $h_1 = h_2 = h_3 = 0$.  To complete the proof we return to \eqref{lin:xit} to deduce that 
\begin{align*}
\| \xi_t \|_{L^2(0,T;H^1(\Omega)')} \anold{\leq C \big ( \| h_1 \|_{L^2(\Sigma)} + \| h_2 \|_{L^2\anold{(0,T)}} + \| h_3 \|_{L^2\anold{(0,T)}} \big ),}
\end{align*}
while if $\beta > 0$, from \eqref{lin:3} we \anold{also} have
\begin{align*}
\| \psi_t \|_{L^2(0,T;H^1(\Omega)')} \anold{\leq C \big ( \| h_1 \|_{L^2(\Sigma)} + \| h_2 \|_{L^2\anold{(0,T)}} + \| h_3 \|_{L^2\anold{(0,T)}} \big ).}
\end{align*}
Lastly after passing to the limit in the Galerkin approximation \anold{we obtain} limit functions $(\xi, \eta, \psi, \bv) \in \bYl$ satisfying \eqref{lin} \anold{except for the $L^2(0,T;H^2_{\bm{n}}(\Omega))$ regularity of $\xi$}.  This can be obtained from viewing \eqref{lin:2} as the variational formulation of the elliptic problem
\begin{align*}
\begin{cases}
- \Lx \xi = \tilde f \an{:}= \eta - \Psi''(\vp) \xi + \chi \psi + \C(\E(\bv) - \xi \E^*): \E^* & \text{ in } \anold{Q}, \\
\pdnu \xi = 0 & \text{ on } \anold{\Sigma},
\end{cases}
\end{align*}
with a right-hand side $\tilde f \in L^2(Q)$, and with the help of elliptic regularity we then infer that $\xi \in L^2(0,T;H^2_{\bm{n}}(\Omega))$.  Hence, \anold{we have shown that $(\xi, \eta, \psi, \bv) \in \bYl$ and this concludes the proof.}
\end{proof}
\subsection{Differentiability of the solution operator}
In this section we establish the Fr\'echet differentiability of the solution operator $\bS$ \an{between suitable Banach spaces,} and \anold{that} the derivative at $\bm{w}^* = (w_1^*, w_2^*, w_3^*) \in \U_{ad}$ in direction $\bm{h} = (h_1, h_2, h_3) \in \U$ is the unique solution $(\xi, \eta, \psi, \bv)$ \an{obtained} from Theorem \ref{thm:lin}.  This is formulated as follows.

\begin{thm}\label{thm:Fdiff}
Under \eqref{ass:beta}-\eqref{ass:opt}, for given $\bm{w}^* \in \U_{ad}$ with $(\vp, \mu, \si, \u) \anold{= \bS(\bm{w}^*)}\in \bY$, the control-to-state operator $\bS$ is Fr\'echet differentiable at $\bm{w}^*$ when viewed as a mapping from $\U$ to $\bX$, where
\begin{align*}
\bX = \begin{cases}
L^\infty(0,T;L^2(\Omega)) \cap L^2(0,T;H^2_{\bm{n}}(\Omega)) \times L^2(Q) & \\
 \times L^\infty(0,T;L^2(\Omega)) \cap L^2(0,T;H^1(\Omega)) \times L^2(0,T;\X) & \text{ if } \beta > 0, \\
& \\
L^\infty(0,T;L^2(\Omega))\cap L^2(0,T;H^2_{\bm{n}}(\Omega))   \times L^2(Q) & \\
\times L^2(0,T;H^1(\Omega)) \times L^2(0,T;\X) & \text{ if } \beta = 0\an{.}
\end{cases}
\end{align*}
Moreover, for all $\bm{h} = (h_1, h_2, h_3) \in \U$, the directional derivative
\begin{align*}
D\bS(\bm{w}^*)[\bm{h}] = (\xi, \eta, \psi, \bv)
\end{align*}
is the unique solution to \eqref{lin} associated to $\bm{h}$.
\end{thm}

\begin{proof}
We denote
\begin{align*}
(\vp_h, \mu_h, \si_h, \u_h) = \bS(\bm{w}^* + \bm{h}),
\end{align*}
and aim to show 
\begin{align*}
\frac{\|\bS(\bm{w}^* + \bm{h}) - \bS(\bm{w}^*) - D\bS(\bm{w}^*)[\bm{h}]\|_{\bX}}{\| \bm{h} \|_{\U}} \to 0 
\qquad \an{\text{as $\| \bm{h} \|_{\U} \to 0$.}}
\end{align*}
This is done via establishing for functions
\begin{align*}
\Phi := \vp_h - \vp - \xi, \quad \lambda := \mu_h - \mu - \eta, \quad \theta := \si_h - \si - \psi, \quad \bz := \u_h - \u - \bv
\end{align*}
the inequality
\begin{align}
\label{toprove}
\| (\Phi, \lambda, \theta, \bz) \|_{\bX} \leq C \| \bh \|_{\U}^2
\end{align}
with a positive constant $C$ independent of $(\Phi, \lambda, \theta, \bz)$ and $\bh$.  To this end, we recall from Theorem \ref{thm:state} and \anold{Theorem} \ref{thm:lin} that the new variables $(\Phi, \lambda, \theta, \bz) \in \bYl$ satisfy
\begin{subequations}\label{Frech:sys}
\begin{alignat}{3}
\label{F:1} 0 & = \inn{\Phi_t}{\zeta} + (\nabla \lambda, \nabla \zeta) +(\lambda_p X_h, \zeta) \\
\notag & \quad - ((\lambda_a + w_2^*) [k(\vp_h) - k(\vp) - k'(\vp)\xi], \zeta) - ((k(\vp_h) - k(\vp)) h_2, \zeta) \\
\label{F:2} 0 & = (\lambda, \zeta) - (\nabla \Phi, \nabla \zeta) - (\Psi'(\vp_h) - \Psi'(\vp) - \Psi''(\vp) \xi, \zeta) \\
\notag & \quad + (\chi \theta, \zeta) - (\C(\E(\bz) - \Phi \E^*): \E^*, \zeta), \\
\label{F:3} 0 & = \beta \inn{\theta_t}{\zeta} + (\nabla \theta, \nabla \zeta) + (B\theta + \lambda_c h(\vp) \theta, \zeta) + (\kappa \theta, \zeta)_{\Gamma}  \\
\notag & \quad + \lambda_c  ((\si - w_3^*) [h(\vp_h) - h(\vp) - h'(\vp)\xi], \zeta) -  \lambda_c ((h(\vp_h) - h(\vp)) h_3, \zeta), \\
\notag & \quad + \lambda_c ((h(\vp_h) - h(\vp))(\si_h - \si), \zeta)  \\
\label{F:4} 0 & = (\C(\E(\bz) - \Phi \E^*), \nabla \bet), 
\end{alignat}
\end{subequations}
for all $\zeta \in H^1(\Omega)$ and $\anold{\bet \in \X}$ and for a.e.~$t \in (0,T)$, where
\begin{align*}
X_h &= \anold{\big(}g(\W_{,\E}(\vp_h, \E(\u_h))) - g(\W_{,\E}(\vp, \E(\u)))\anold{\big{)}}  \\
& \qquad \times \Big [ \anold{(f(\vp_h) - f(\vp))(\si_h - \si) +  f(\vp)(\si_h - \si)} + (f(\vp_h) - f(\vp)) \si \Big ] \\
& \quad + g(\W_{,\E}(\vp, \E(\u))) \\
& \qquad \times \Big [ (f(\vp_h) - f(\vp))(\si_h - \si) + \si \anold{[}f(\vp_h) - f(\vp) - f'(\vp)\xi\anold{]} + f(\vp)\anold{\theta}\Big ] \\
& \quad + f(\vp) \si \Big [ g(\W_{,\E}(\vp_h, \E(\u_h))) - g(\W_{,\E}(\vp, \E(\u)))  - g'(\W_{,\E}(\vp, \E(\u))) : \anold{\C(\E(\bv)-\xi \E^* )}\Big ].
\end{align*}
\anold{Take note that $h_1$ does not appear in \eqref{Frech:sys} since the state system \eqref{state} is linear in $w_1^* = \sigma_B$.} We invoke Taylor's theorem with integral remainder for $f \in W^{2,\infty}(\R)$:
\begin{align*}
f(x) - f(a) - f'(a)(x-a) = (x-a)^2 \int_0^1 f''(a + z(x-a))(1-z) \dz \quad \text{ for } a, x \in \R,
\end{align*}
to deduce that 
\begin{align*}
f(\vp_h) - f(\vp) - f'(\vp) \xi & = f'(\vp) \Phi + (\vp_h - \vp)^2 R_f, \\
h(\vp_h) - h(\vp) - h'(\vp) \xi & = h'(\vp) \Phi + (\vp_h - \vp)^2 R_h, \\
k(\vp_h) - k(\vp) - k'(\vp) \xi & = k'(\vp) \Phi + (\vp_h - \vp)^2 R_k, \\ 
\Psi'(\vp_h) - \Psi'(\vp) - \Psi''(\vp) \xi & = \Psi''(\vp) \Phi + (\vp_h - \vp)^2 \an{\an{R_{\Psi}}},
\end{align*}
where
\begin{align*}
R_f & = \int_0^1 f''(\vp + z(\vp_h - \vp)) (1-z) \dz, \quad R_h = \int_0^1 h''(\vp + z(\vp_h - \vp)) (1-z) \dz, \\
R_k & = \int_0^1 k''(\vp + z(\vp_h - \vp)) (1-z) \dz, \quad \an{R_{\Psi}} = \int_0^1 \Psi'''(\vp + z(\vp_h - \vp))(1-z) \dz,
\end{align*}
and in light of the \an{regularity} assumption \an{\eqref{ass:opt}}, as well as \eqref{ass:pot} and \eqref{state:bdd}, there exists a positive constant $C$ such that 
\begin{align}\label{remain}
\|R_f\|_{L^\infty} \an{+ \|R_h \|_{L^{\infty}} + \|R_k \|_{L^\infty} }\leq C, \quad \|\an{R_{\Psi}}\|_{L^6} \leq C\big ( 1 + \|\vp \|_{L^6} + \|\vp_h \|_{L^6} \big ) \leq C.
\end{align}
Next, we test \eqref{F:1} with $\Phi$, \eqref{F:2} with $\lambda$, \eqref{F:3} with $M \theta$ and \eqref{F:4} with $K \bz$ for positive constants $M, K$ yet to be determined, and upon adding the resulting equations we arrive at
\begin{equation}\label{F:est}
\begin{aligned}
& \frac{1}{2}\frac{d}{dt} \Big ( \| \Phi \|_{L^2}^2 + \beta M \| \theta \|_{L^2}^2 \Big )  + \| \lambda \|_{L^2}^2 + M \| \nabla \theta \|^2 \\
& \qquad + M(B \theta + \lambda_c h(\vp) \theta, \theta) + M \kappa \| \theta \|_{L^2_\Gamma}^2 + K c_0 \| \E(\bz) \|^2 \\
& \quad \leq -(\lambda_p X_h, \Phi) + \big((\lambda_a + w_2^*) [k(\vp_h) - k(\vp) - k'(\vp) \xi] + h_2 (k(\vp_h) - k(\vp)), \Phi\big) \\
& \qquad + (\Psi'(\vp_h) - \Psi'(\vp) - \Psi''(\vp)\xi, \lambda) - (\chi \theta, \lambda) + (\C(\E(\bz) - \Phi \E^*) : \E^*, \lambda) \\
& \qquad - M \lambda_c ((\si - w_3^*) [h(\vp_h) - h(\vp) - h'(\vp)\xi], \theta) \anold{+} M \lambda_c ((h(\vp_h) - h(\vp)){h_3} , \theta) \\
& \qquad - M \lambda_c ((h(\vp_h) - h(\vp))(\si_h - \si), \theta) + K(\C(\Phi \E^*)\an{,} \E(\bz)).
\end{aligned}
\end{equation}
Looking at the terms involving $\bz$, we see 
\begin{align*}
(\C(\E(\bz) - \Phi \E^*) : \E^*, \lambda) + K (\C(\Phi \E^*) \an{,} \E(\bz)) \leq \frac{1}{8} \| \lambda \|_{L^2}^2 + C \| \E(\bz) \|_{L^2}^2 + C(1+K^2) \| \Phi \|_{L^2}^2
\end{align*}
for a positive constant $C$ independent of $K$. Next, for terms involving $\theta$, we employ the boundedness of $\si$ and $w_3^*$, and the Lipschitz continuity of $h$ to obtain 
\begin{align*}
& - (\chi \theta, \lambda)  - M \lambda_c ((h(\vp_h) - h(\vp))(\si_h - \si), \theta) \\
& \qquad  - M \lambda_c ((\si - w_3^*) [h(\vp_h) - h(\vp) - h'(\vp)\xi], \theta) \anold{+}  M \lambda_c ((h(\vp_h) - h(\vp)) h_3, \theta) \\
& \quad \leq \frac{1}{8} \| \lambda \|_{L^2}^2 + C \| \theta \|_{L^2}^2 + M^2 C \| \vp_h - \vp \|_{L^4}^2 \| \si_h - \si \|_{L^4}^2 \\
& \qquad  \anold{+ M^2C \| h''(\vp) \Phi + (\vp_h - \vp)^2 R_h \|_{L^2}^2 + M^2 C \| \vp_h - \vp \|_{L^2}^2 |h_3|^2}\\
& \quad \leq \frac{1}{8} \| \lambda \|_{L^2}^2 + C \| \theta \|_{L^2}^2 + CM^2 \big ( \| \Phi \|_{L^2}^2 + \| \vp_h - \vp \|_{H^1}^2 \big ( \| \si_h - \si \|_{H^1}^2 + \| \vp_h - \vp \|_{H^1}^2 + |h_3|^2 \big )\big )
\end{align*}
for a positive constant $C$ independent of $M$.   Next, for the terms involving $k$ in \eqref{F:est} we similarly have
\begin{align*}
& ((\lambda_a + w_2^*) [k(\vp_h) - k(\vp) - k'(\vp) \xi] + h_2 (k(\vp_h) - k(\vp)), \Phi) \\
& \quad \leq C \| \Phi \|_{L^2}^2 \anold{+ C \| \vp_h - \vp \|_{L^4}^4+ C \| \vp_h - \vp \|_{L^2}^2 | h_2|^2 } \\
& \quad \leq C \|\Phi \|_{L^2}^2 + C \big ( \anold{|h_2|^2 + \| \vp_h - \vp \|_{H^1}^2 } \big ) \| \vp_h - \vp \|_{H^1}^2,
\end{align*}
and for the terms involving $\Psi'$, 
\begin{align*}
& (\Psi'(\vp_h) - \Psi'(\vp) - \Psi''(\vp)\xi, \lambda) = (\Psi''(\vp)\Phi + \an{R_{\Psi}} (\vp_h - \vp)^2, \lambda)\\
& \quad \leq \| \Psi''(\vp) \|_{L^\infty} \| \Phi \|_{L^2} \| \lambda \|_{L^2} + \| \an{R_{\Psi}} \|_{L^6} \| \vp_h - \vp \|_{L^6}^2 \| \lambda \|_{L^2} \\
& \quad \leq \frac{1}{8} \| \lambda \|_{L^2}^2 + C\big ( 1+ \| \vp \|_{H^2}^4 \big ) \| \Phi \|_{L^2}^2 + C \|\vp_h - \vp \|_{H^1}^4,
\end{align*}
where we used \eqref{ass:pot}, \eqref{state:ctsdep} and \eqref{remain}.  Lastly, we tackle the term involving $X_h$.  First, \todo{we observe with the assumption $g \in W^{2,\infty}(\R^{d \times d}, \R)$ from \eqref{ass:opt} that }
\begin{align*}
& |g(\W_{,\E}(\vp_h, \E(\u_h))) - g(\W_{,\E}(\vp, \E(\u)))| \\
& \quad  = \left | \int_0^1 g'(z \W_{,\E}(\vp_h, \E(\u_h)) + (1-z) \W_{,\E}(\vp, \E(\u))) \dz : \C(\E(\u_h - \u) - (\vp_h-\vp)\E^*) \right | \\
& \quad \leq C|\an{\E(\u_h) - \E(\u)}| +  C|\vp_h - \vp|,
\end{align*}
and
\begin{align*}
& g(\W_{,\E}(\vp_h, \E(\u_h))) - g(\W_{,\E}(\vp, \E(\u)))  - g'(\W_{,\E}(\vp, \E(\u))) : \anold{\C(\E(\bv)-\xi \E^* )} \\
& \quad = g'(\W_{,\E}(\vp, \E(\u))) : \anold{\C(\E(\bz)-\Phi \E^* )} \\
& \qquad + \todo{\Big (\int_0^1 (1-z) g''(\cdot) \dz \Big ) [\C(\E(\u_h - \u) - (\vp_h - \vp)\E^*)] : [\C(\E(\u_h - \u) - (\vp_h - \vp)\E^*)]} \\
& \quad \leq g'(\W_{,\E}(\vp, \E(\u))) : \anold{\C(\E(\bz)-\Phi \E^* )} + C|\C(\E(\u_h - \u) - (\vp_h - \vp)\E^*)|^2,
\end{align*}
\todo{where the fourth order tensor $g''(\cdot)$ is evaluated at $(1-z)\W_{,\E}(\vp, \E(\u)) + z \W_{,\E}(\vp_h, \E(\u_h))$.}  Hence, using the boundedness of $f(\vp)$, \todo{$g(\cdot)$, $g'(\cdot)$, $g''(\cdot)$}, $\si$ and $\si_h$ that are independent of $\| \bm{h} \|_{\U}$, we obtain for a positive constant $\eps > 0$ to be determined later,
\begin{align*}
(\lambda_p X_h, \Phi) & \leq \eps \| \Phi \|_{H^2}^2 + C \| X_h \|_{L^1}^2 \\
& \leq \eps \| \Phi \|_{H^2}^2 + C \an{\| \E(\u_h) - \E(\u) \|_{L^2}^2\| \vp_h - \vp \|_{L^4}^2\| \si_h - \si \|_{L^4}^2  }
\\ & \quad  +\an{ C \| \vp_h - \vp \|_{L^4}^4\| \si_h - \si \|_{L^2}^2 }
\an{ + C \an{\| \E(\u_h) - \E(\u) \|_{L^2}^2 } \big ( \| \si_h - \si \|_{L^2}^2 + \| \vp_h - \vp \|_{L^2}^2\big )} \\
& \quad + C \| \vp_h - \vp \|_{L^2}^4 + C \anold{ \| \vp_h - \vp \|_{L^2}^2 \| \si_h - \si \|_{L^2}^2} + C \| \Phi \|_{L^2}^2 \\
& \quad  + C \| \theta \|_{L^2}^2 + C \| \E(\bz) \|_{L^2}^2 + C \| \E(\u_h) - \E(\u) \|_{L^2}^4
\\ & \leq \an{\eps \| \Phi \|_{H^2}^2 +  C  (\| \bm{h} \|_{\U}^4 + \| \bm{h} \|_{\U}^6) }.
\end{align*}
To close the estimate we require an estimate for $\| \Phi \|_{H^2}^2$, which can be obtained from \eqref{F:2}.  Using that $(\Phi, \lambda, \theta, \anold{\bz}) \in \bYl$ we see that 
\begin{align*}
\| \Lx \Phi \|_{L^2}^2 & \leq C \| \lambda \|_{L^2}^2 + C \| \Psi'(\vp_h) - \Psi'(\vp) - \Psi''(\vp) \xi \|_{L^2}^2 + C \| \theta \|_{L^2}^2 + C \| \E(\bz) \|_{L^2}^2 + C \| \Phi \|_{L^2}^2 \\
& \leq C \big ( 1 + \| \vp \|_{H^2}^4 \big ) \| \Phi \|_{L^2}^2 + C \| \bm{h} \|_{\U}^4 + C \big ( \| \theta \|_{L^2}^2 + \| \E(\bz) \|_{L^2}^2 + \| \lambda \|_{L^2}^2 \big ).
\end{align*}
By elliptic regularity, there exists a positive constant $C_0$ independent of $(\Phi, \lambda, \theta, \bz)$ and $\bm{h}$, as well as $M$ and $K$, such that 
\begin{align}\label{F:H2}
\| \Phi \|_{H^2}^2 \leq C \big ( 1 + \| \vp \|_{H^2}^4 \big ) \| \Phi \|_{L^2}^2 + C \| \bm{h} \|_{\U}^4 + C_0 \big ( \| \theta \|_{L^2}^2 + \| \E(\bz) \|_{L^2}^2 + \| \lambda \|_{L^2}^2 \big ).
\end{align}
Let $\alpha$ be a positive constant such that 
\begin{align*}
\alpha C_0 \leq \tfrac{1}{8}.
\end{align*}
Multiplying \eqref{F:H2} with $\alpha$ and adding to \eqref{F:est}, then employing the estimates for the right-hand side and choosing $\eps = \frac{\alpha}{2}$, we obtain
\begin{align*}
& \frac{1}{2} \frac{d}{dt} \Big ( \| \Phi \|_{L^2}^2 + \beta M \| \theta \|_{L^2}^2 \Big ) + \frac{1}{2} \| \lambda \|_{L^2}^2 + \frac{\alpha}{2} \| \Phi \|_{H^2}^2 \\
& \qquad  + M \| \nabla \theta \|_{L^2}^2 + MB \| \theta \|_{L^2}^2 + M\kappa \| \theta \|_{L^2_\Gamma}^2 - \an{\hat C} \| \theta \|_{L^2}^2 + (K c_0 - \an{\hat C}) \| \E(\bz) \|_{L^2}^2\\
& \quad \leq C \big ( 1 + \| \vp \|_{H^2}^4 \big ) \| \Phi \|_{L^2}^2 + C \| \vp_h - \vp \|_{H^1}^2 \big ( |h_2|^2 + |h_3|^2 \big ) \\
& \qquad + C \| \vp_h - \vp \|_{H^1}^2 \big ( \| \si_h - \si \|_{H^1}^2 + \| \vp_h - \vp \|_{H^1}^2 + \| \E(\u_h) - \E(\u) \|_{L^2}^2 \big ) \\
& \qquad + C \| \E(\u_h) - \E(\u) \|_{L^2}^2 \big ( \| \E(\u_h) - \E(\u) \|_{L^2}^2 + \| \si_h - \si \|_{L^2}^2 \big ) 
\an{+  C  (\| \bm{h} \|_{\U}^4 + \| \bm{h} \|_{\U}^6)}
\\
& \quad =: C \big ( 1 + \| \vp \|_{H^2}^4 \big ) \| \Phi \|_{L^2}^2 + C \| \vp_h - \vp \|_{H^1}^2 \big (|h_2|^2 +|h_3|^2 \big ) + \mathcal{R}_h,
\end{align*}
where the positive constants $\an{\hat C}$ appearing on the left-hand side are independent of $M$ and $K$.  Hence, choosing $M$ and $K$ sufficiently large, with Gronwall's inequality and Korn's inequality, as well as $\Phi(0) = \theta(0) = 0$, we have
\begin{align*}
& \| \Phi \|_{L^\infty(0,T;L^2)}^2 + \beta \| \theta \|_{L^\infty(0,T;L^2)}^2 + \| \lambda \|_{L^2(Q)}^2 + \| \Phi \|_{L^2(0,T;H^2)}^2 \\
& \qquad + \| \theta \|_{L^2(0,T;H^1)}^2 + \| \bz \|_{L^2(0,T;\X)}^2 \\
& \quad \leq C \exp \Big ( C + C \| \vp \|_{L^4(0,T;H^2)}^4  \Big ) \int_0^T  \| \vp_h - \vp \|_{H^1}^2 \big (|h_2|^2 + |h_3|^2 \big ) + \mathcal{R}_h \dt \\
& \quad \leq \an{  C  (\| \bm{h} \|_{\U}^4 + \| \bm{h} \|_{\U}^6)},
\end{align*}
where the last inequality comes from the application of \eqref{state:ctsdep}.  This completes the proof \an{as \eqref{toprove} has been shown}.
\end{proof}

\subsection{Adjoint system}
Associated to an optimal control $\bm{w}^* \in \U_{ad}$ and its corresponding solution $(\vp, \mu, \si, \u)$ are the adjoint variables $(p,q,r,\bs)$ that satisfies the following adjoint system written in strong form:
\begin{subequations}\label{adj}
\begin{alignat}{3}
\anold{{f}_1} & = \anold{- p_t} - \Lx q + \anold{\G} && \quad\text{ in } Q, \\
q & = \anold{-\Lx p} && \quad\text{ in } Q, \\
0 & = - \beta r_t - \Lx r + B r + \anold{\K} &&\quad \text{ in } Q, \\
\anold{{\bm{f}}_2} & = \div (\C(\E(\bs) +\anold{\bH})) && \quad\text{ in } Q, \\
p(T)  &= \alpha_\Omega (\vp(T) - \vp_\Omega), \quad r(T) = 0 && \quad\text{ in } \Omega, \\
0 & = \pdnu p \anold{= \pdnu q}, \quad \pdnu r + \kappa r = 0 && \quad\text{ on } \Sigma, \\
\an{\0} & = (\C(\E(\bs) + \anold{\bH) - \bm{f}_2}) \bm{n} && \quad\text{ on } \anold{\SD},
\\ 
\anold{\bs} & =\anold{ \0} && \quad \anold{\text{ on } \SN,}
\end{alignat}
\end{subequations}
where using the notation $n'(\cdot,\vp) = \frac{\pd n}{\pd \vp} (\cdot,\vp)$, 
\begin{align*}
\anold{{f}_1 } & = \anold{f_1(\vp,\E(\u))} = \anold{\alpha_Q (\vp - \vp_Q)+  \tfrac{\alpha_{\E}}{2} n'(\cdot,\vp)|\W_{,\E}(\vp, \E(\u))|^2 }
\\ & \quad
\anold{ - \alpha_{\E} n(\cdot,\vp) \W_{,\E}(\vp,\E(\u)) : \C \E^*,}\\
\anold{{\bm{f}}_2} & = \anold{\an{{\bm f}_2(\vp,\E(\u))}} =\anold{-\div(\alpha_{\E} n(\cdot,\vp) \C \W_{,\E}(\vp, \E(\u))),}
\\ 
\anold{\G } & = \G(\vp,\si,\E(\u),p,q,r,\anold{\E(\bs),}w_2^*\anold{,w_3^*}) = \Psi''(\vp) q \anold{+} (\lambda_a + w_2^*) k'(\vp) p + q \C \E^* : \E^*  \\
& \quad \anold{-} \lambda_p \si p \big ( f'(\vp) g(\W_{,\E}(\vp, \E(\u))) \anold{+} f(\vp) g'(\W_{,\E}(\vp, \E(\u))) : \C \E^* \big ) \\
& \quad  \anold{+} h'(\vp)(\lambda_c \si - \anold{w_3^*}) r + \C \E^* : \E(\bs)\anold{,}  
\\
\anold{\K} & = \K(\vp,\E(\u),p,q,r) =  h(\vp) \lambda_c r - \lambda_p f(\vp) g(\W_{,\E}(\vp, \E(\u))) p \anold{-} \chi q, \\
\anold{\bH} & = \bH(\vp, \si, \E(\u), p, q) = \anold{q \E^*+\lambda_p p \si f(\vp) g'(\W_{,\E}(\vp, \E(\u))) .}
\end{align*} 
We introduce the solution space
\begin{align*}
\bZ = \begin{cases}
H^1(0,T;H^2_{\bm{n}}(\Omega)') \cap L^2(0,T;H^2_{\bm{n}}(\Omega)) \times L^2(Q) & \\
\times H^1(0,T;H^1(\Omega)') \cap L^2(0,T;H^1(\Omega)) \times L^2(0,T;\X) & \text{ if } \beta > 0, \\
& \\
H^1(0,T;H^2_{\bm{n}}(\Omega)') \cap L^2(0,T;H^2_{\bm{n}}(\Omega)) \times L^2(Q) & \\
\times L^2(0,T;H^1(\Omega)) \times L^2(0,T;\X) & \text{ if } \beta = 0.
\end{cases}
\end{align*}

\begin{thm}\label{thm:adj}
For given $\bm{w}^* \in \U_{ad}$ with $(\vp, \mu, \si, \u) = \bS(\bm{w}^*) \in \bY$, under \eqref{ass:beta}-\an{\eqref{ass:target}}, there exists a unique solution $(p,q,r,\bs) \in \bZ$ to \an{the adjoint system} \eqref{adj} with $p(T) = \alpha_\Omega (\vp(T) - \vp_\Omega)$, $r(\an{T}) = 0$ if $\beta > 0$, and
\begin{subequations}\label{adj:weak}
\begin{alignat}{3}
0  &= \int_0^T \anold{-}\inn{p_t}{\zeta}_{H^2} - (q, \Lx \zeta)+  ( \anold{\G-{f}_1, \zeta)  \dt,} \label{adj:1} \\
0  &= \int_0^T \anold{-}(q, \phi) + (\nabla p, \nabla \phi) \dt, \label{adj:2} \\
0  &= \int_0^T \beta \inn{-r_t}{\phi} + (\nabla r, \nabla \phi) + (Br, \phi) + \kappa (r, \phi)_\Gamma + (\anold{\K},\phi)\dt, \label{adj:3} \\
0  & \anold{= \int_0^T (\C(\E(\bs) + \anold{\bH}) , \nabla \bet) -  (\bm{f}_2 , \bet)} \dt, \label{adj:4}
\end{alignat}
\end{subequations}
for all $\zeta \in L^2(0,T;H^2_{\bm{n}}(\Omega))$, $\phi \in L^2(0,T;H^1(\Omega))$ and $\bet \in L^2(0,T;\X)$.
\end{thm}

\anold{
\begin{remark}
Let us notice that the test function space $H^2_{\bm n}(\Omega)$ in \eqref{adj:1} can be weakened by assuming a more regular target function $\vp_\Omega$. In fact, formally testing \eqref{adj:1} by $q$ and \eqref{adj:2} by $p_t$ will lead to the regularity $q\in L^2 (0,T;{H^1(\Omega)})$, and $ p\in L^\infty (0,T;{H^1(\Omega))}$ provided  $p(T)=\alpha_\Omega(\vp(T)-\vp_\Omega) \in H^1(\Omega)$, which is fulfilled if $\vp_\Omega \in H^1(\Omega)$. 
\end{remark}}

\begin{proof}
We proceed with formal estimates that can be rigorously derived with a standard Galerkin approximation \anold{and let us note that in} the following positive constants denoted by the symbol $C$ will be independent of the Galerkin parameter.  Then, testing $\zeta = Kp$ in \eqref{adj:1}, $\phi = -Kq$ and $\phi = p$ in \eqref{adj:2}, $\phi = Hr$ in \eqref{adj:3} and $\bet =  \anold{Z}\bs$ in \eqref{adj:4} for some positive constants $K$, $H$ and $\anold{Z}$ yet to be determined, we obtain after summing the resulting equalities
\begin{equation}\label{adj:est}
\begin{aligned}
& - \frac{1}{2} \frac{d}{dt} \Big (K \| p \|_{L^2}^2 + H \beta \| r \|_{L^2}^2 \Big ) + K\| q \|_{L^2}^2 + \| \nabla p \|_{L^2}^2 \\
& \qquad + H \| \nabla r \|_{L^2}^2 + H \kappa \| r \|_{L^2_\Gamma}^2 + H B \| r \|_{L^2}^2 + \anold{Z} c_0 \| \E(\bs) \|_{L^2}^2 \\
& \quad \leq \anold{-(K(\G - {f}_1), p)} \an{+} \anold{(q,p)} - (H\K, r) - \anold{\an{(Z (\C\bH - {\bm{f}}_2)} , \E(\bs))}.
\end{aligned}
\end{equation}
Firstly, we obtain from \eqref{adj:2} and elliptic regularity that
\begin{align}\label{adj:reg}
\| p \|_{H^2}^2 \leq C \| \Lx p \|_{L^2}^2 + C \| p \|_{L^2}^2 \leq C \| q \|_{L^2}^2 + C \|p \|_{L^2}^2.
\end{align}
Then, a short calculation shows that 
\begin{align*}
\anold{(K(\G - {f}_1), p)} & \leq  \| q \|_{L^2}^2 + C \big ( 1 + \| \Psi''(\vp) \|_{L^\infty}^2 \big ) \| p \|_{L^2}^2 + C \| r \|_{L^2}^2  \\
& \quad+ C \| \vp \|_{L^2}^2 + C \| \E(\bs) \|_{L^2}^2 + \tfrac{1}{2} \| p \|_{H^2}^2 + C\| \vp - \vp_Q \|_{L^2}^2, \\
\anold{(q,p)} & \leq  \| q \|_{L^2}^2 + C \| p \|_{L^2}^2, \\
(H \K, r) & \leq H^2 \| q \|_{L^2}^2 + C \big ( \| r \|_{L^2}^2 + \| p \|_{L^2}^2 \big ), \\
{\an{(Z (\C\bH - {\bm{f}}_2)}, \E(\bs))} & \leq \anold{Z}^2 \| q \|_{L^2}^2 + C\anold{Z}^2 \big ( \| p \|_{L^2}^2 + \| \vp \|_{L^2}^2 + \| \E(\u) \|_{L^2}^ 2 \big ) + C \| \E(\bs) \|_{L^2}^2,
\end{align*}
with positive constants $C$ independent of $H$ and $\anold{Z}$.  Adding \eqref{adj:reg} to \eqref{adj:est} and substituting the above yields
\begin{align*}
& - \frac{1}{2} \frac{d}{dt} \Big (K \| p \|_{L^2}^2 + H \beta \| r \|_{L^2}^2 \Big ) + (K - (1+C+H^2+\anold{Z}^2))\| q \|_{L^2}^2 + \tfrac{1}{2} \| p \|_{H^2}^2 \\
& \qquad + H \| \nabla r \|_{L^2}^2 + H \kappa \| r \|_{L^2_\Gamma}^2 + (H B - C) \| r \|_{L^2}^2 + (\anold{Z} c_0 - C) \| \E(\bs) \|_{L^2}^2 \\
& \quad \leq C \big ( 1 + \| \vp \|_{H^2}^4 \big ) \| p \|_{L^2}^2 + C \| \vp - \vp_Q \|_{L^2}^2   + C \big (  \| \vp \|_{L^2}^2 + \| \E(\u) \|_{L^2}^ 2 \big ).
\end{align*}
If $B > 0$, we choose $HB > C$, otherwise we use the generalised Poincar\'e inequality with $H$ sufficiently large so that 
\begin{align*}
H \| \nabla r \|_{L^2}^2 + H \kappa \| r \|_{L^2_\Gamma}^2 \geq (C+1) \| r \|_{L^2}^2.
\end{align*}
Then, choosing $\anold{Z}$ sufficiently large so that $\anold{Z}c_0 > C$, and then finally $K$ sufficiently large, we obtain via Gronwall's inequality (applied backwards in time) and Korn's inequality that 
\begin{equation}\label{adj:est1}
\begin{aligned}
& \| p \|_{L^\infty(0,T;L^2)}^2 + \beta \| r \|_{L^\infty(0,T;L^2)}^2 + \| q \|_{L^2(Q)}^2 \\
& \qquad + \| p \|_{L^2(0,T;H^2)}^2 + \| r \|_{L^2(0,T;H^1)}^2 + \| \bs \|_{L^2(0,T;\X)}^2 \\
& \quad \leq C \|\vp - \vp_Q \|_{L^2(Q)}^2 + C \| \vp \|_{L^2(Q)}^2 + C \| \E(\u) \|_{L^2(Q)}^2.
\end{aligned}
\end{equation}
Then, from \eqref{adj:1} we infer
\begin{align*}
\| p_t \|_{L^2(0,T;H^2_{\bm{n}}(\Omega)')} \leq  C \big ( 1 +  \| \Psi''(\vp) \|_{L^\infty(0,T;L^2)} \big ) \| q \|_{L^2(Q)} + C \| p \|_{L^2(Q)} + C \| r \|_{L^2(Q)},
\end{align*}
and if $\beta > 0$, \anold{a comparison of terms in} \eqref{adj:3} gives
\begin{align*}
\| r_t \|_{L^2(0,T;H^1(\Omega)')} \leq C\| r \|_{L^2(0,T;H^1)} + C \| p \|_{L^2(Q)} + C \| q \|_{L^2(Q)}.
\end{align*}
These estimates are sufficient to pass to the limit and deduce the existence of a solution $(p,q,r,\bs) \in \bZ$ to \eqref{adj} in the sense that \eqref{adj:weak} is fulfilled. Moreover, as the adjoint system is linear in $(p,q,r,\bs)$, the difference of any two solutions satisfy \eqref{adj:weak} where the terms involving $n(\cdot,\vp)$ and $\vp - \vp_Q$ are absent in ${f}_1$ and $\bm{f}_2$.  Consequently, we arrive at an analogue of \eqref{adj:est1} for the difference of two solutions where the right-hand side is zero, which in turn leads to uniqueness of solutions.
\end{proof}

\subsection{Optimality conditions}
\anold{Lastly, we exploit the differentiability property of $\bS$ established so far to obtain the first-order necessary conditions for optimality.
In this direction, we} first express the reduced cost functional $\J$ as the sum
\begin{align*}
\J(\bm{w}) \anold{:=} \J_1(\bm{w}) + \J_2(\bm{w}),
\end{align*}
where
\begin{align*}
\J_1 (\bm{w}) & = \anold{\frac{\alpha_\Omega}{2} \| \bS_1(\bm{w}) - \vp_\Omega \|_{L^2(\Omega)}^2 +\frac{\alpha_Q}{2} \| \bS_1(\bm{w}) - \vp_Q \|_{L^2(Q)}^2 }\\
& \quad + \frac{\alpha_{\E}}{2} \int_Q n(x, \bS_1(\bm{w})) |\W_{,\E}(\bS_1(\bm{w}), \E(\bS_4(\bm{w})))|^2 \dx \dt \\
& \quad + \frac{\gamma_1}{2} \| w_1 \|_{L^2(\Sigma)}^2 + \frac{\gamma_2}{2} \| w_2 \|_{L^2(0,T)}^2 + \frac{\gamma_3}{2} \| w_3 \|_{L^2(0,T)}^2, \\
\J_2(\bm{w}) & = \gamma_4 \| w_2 \|_{L^1(0,T)}+ \gamma_5 \| w_3 \|_{L^1(0,T)}.
\end{align*}
Then, for arbitrary $\bm{y} \in \U_{ad}$ and an optimal control $\bm{w}^* \in \U_{ad}$ with corresponding state $(\vp, \mu, \si, \u) = \bS(\bm{w}^*) \in \bY$ and linearised state variables $(\xi, \eta, \psi, \bv) \in \bYl$ to \eqref{lin:weak} corresponding to $\bm{h} = \bm{y} - \bm{w}^*$,
the differentiability of the solution operator $\bS: \U \to \bY$ and the chain rule shows that 
\begin{equation}\label{opt:0}
\begin{aligned}
D \J_1(\bm{w}^*)[\bm{h}] & = D \J_1(\bm{w}^*)[\bm{y} - \bm{w}^*] \\
& = \anold{ \int_\Omega \alpha_\Omega (\vp(T) - \vp_\Omega)\xi(T) \dx + \int_Q \alpha_Q  (\vp - \vp_Q)\xi  \dx \dt}\\
& \qquad + \alpha_{\E} \int_Q \tfrac{1}{2} n'(x,\vp) \xi |\W_{,\E}|^2 + n(x,\vp) \W_{,\E} : \C(\E(\bv) - \xi \E^*) \dx \dt \\
& \qquad +  \int_0^T \gamma_1( w_1^* , h_1)_\Gamma \dt + \int_0^T \gamma_2 w_2^* h_2 + \gamma_3 w_3^* h_3 \dt,
\end{aligned}
\end{equation}
where $\W_{,\E}$ is evaluated at $(\vp, \E(\u))$.  On the other hand, optimality of $\bm{w}^*$ and the convexity of $\J_2$ leads to 
\begin{align*}
0 & \leq \J(\bm{w}^* + t (\bm{y}- \bm{w}^*)) - \J(\bm{w}^*) \\
& = \J_1(\bm{w}^* + t(\bm{y} - \bm{w}^*)) - \J_1(\bm{w}^*) + \J_2((1-t)\bm{w}^* + t\bm{y}) - \J_2(\bm{w}^*) \\
& \leq \J_1(\bm{w}^* + t(\bm{y} - \bm{w}^*)) - \J_1(\bm{w}^*) + t [\J_2(\bm{y}) - \J_2( \bm{w}^*)]
\end{align*}
for all $t \in (0,1)$ and arbitrary $\bm{y} \in \U_{ad}$.  Dividing by $t$ and passing to the limit $t \to 0$ yields the inequality
\begin{align}\label{opt:1}
0 \leq D \J_1(\bm{w}^*)[\bm{y} - \bm{w}^*] + \J_2(\bm{y})  - \J_2(\bm{w}^*) \quad \forall \bm{y} \in \U_{ad}.
\end{align}
Arguing as in \cite[Sec.~3 and 4]{ST}, the inequality \eqref{opt:1} allows us to interpret $\bm{w}^*$ as a solution to the convex minimisation problem
\anold{\begin{align*}
\min_{\bm{y} \in \U} \Big ( D \J_1(\bm{w}^*) \an{[\bm{y}]} + \J_2(\bm{y}) + \mathbb{I}_{\U_{ad}}(\bm{y}) \Big ) \quad \text{ where } \quad\mathbb{I}_{\U_{ad}}(\bm{y}) = \begin{cases} 0 & \text{ if } \bm{y} \in \U_{ad}, \\
+\infty &\text{ otherwise},
\end{cases}
\end{align*}}
denotes the indicator function of the set $\U_{ad}$.  Using the definition of subdifferentials, the inequality \eqref{opt:1} can also be interpreted as
\begin{align*}
\0 \in \pd \Big ( D \J_1(\bm{w}^*) + \J_2 + \mathbb{I}_{\U_{ad}} \Big )( \bm{w}^*) = \{D\J_1(\bm{w}^*)\} + \pd \J_2(\bm{w}^*) + \pd \mathbb{I}_{\U_{ad}}(\bm{w}^*),
\end{align*}
where the equality is due to the well-known sum rule for subdifferentials of convex functionals.  This implies there exist elements $\bm{\zeta} \in \pd \mathbb{I}_{\U_{ad}}(\bm{w}^*)$, and $\lambda_2(t) \in \pd \| w_2^*(t) \|_{L^1(0,T)}$, $\lambda_3(t) \in \pd \| w_3^*(t) \|_{L^1(0,T)}$ with $\lambda_2, \lambda_3 \in L^\infty(0,T)$ satisfy \eqref{lambd} for a.e.~$t \in (0,T)$, \an{see, e.g.,} \cite[Sec.~4.2]{ST} for similar ideas regarding the derivation, such that 
\begin{align*}
\0 = D \J_1(\bm{w}^*) + \bm{\lambda} + \bm{\zeta}
\end{align*}
for $\bm{\lambda} = (0, \gamma_4 \lambda_2, \gamma_5 \lambda_3)^{\top}$.  From the definition of $\pd \mathbb{I}_{\U_{ad}}$ we have
\begin{align*}
(\bm{\zeta}, \bm{y} - \bm{w}^*) \leq \mathbb{I}_{\U_{ad}}(\bm{y}) - \mathbb{I}_{\U_{ad}}(\bm{w}^*)  = 0 \quad \text{ as } \quad \bm{y}, \bm{w}^* \in \U_{ad},
\end{align*}
where we use $(\cdot,\cdot)$ to denote the inner product on $\U$.  Hence, from \eqref{opt:1} we deduce that $\bm{w}^* \in \U_{ad}$ satisfies
\begin{align}\label{opt:2}
0 \leq D\J_1(\bm{w}^*)[\bm{y} - \bm{w}^*] + (\bm{\lambda}, \bm{y} - \bm{w}^*) \quad \forall \bm{y} \in \U_{ad}.
\end{align}
Next, we aim to simplify \eqref{opt:0} with the help of the adjoint variables.  The standard procedure is to test \eqref{lin:1} with $\zeta = p$, \eqref{lin:2} with $\zeta = \an{-}q$, \eqref{lin:3} with $\zeta = r$, \eqref{lin:4} with $\bet = \bs$, then take the sum and compare with the resulting equality obtained from the sum of \eqref{adj:1} with $\zeta = \xi$, \eqref{adj:2} with $\phi = \eta$, \eqref{adj:3} with $\phi = \psi$ and \eqref{adj:4} with $\bet = -\bv$, which yields the relations
\begin{align*}
& \int_0^T \kappa (h_1, r)_\Gamma + (h_3 h(\vp), r) - (h'(\vp)\xi (\lambda_c \si - w_3^*), r) \dt \\
& \quad = \int_0^T (\lambda_p f(\vp)g(\W_{,\E}(\vp, \E(\u))\anold{)} p, \psi) - (\chi q, \psi) \dt,
\end{align*}
and
\begin{align*}
& \anold{\alpha_Q} \int_Q  (\vp - \vp_Q) \xi \dx \dt + \anold{\alpha_\Omega} \int_\Omega  (\vp(T) - \vp_\Omega) \xi(T) \dx \\
& \qquad + \int_Q {\frac {\alpha_{\E}} {2}}  n'(x,\vp) \xi |\W_{,\E}(\vp, \E(\u))|^2 + \alpha_{\E} n(x,\vp) \W_{,\E}(\vp, \E(\u)) : \C(\E(\bv) - \xi \E^*) \dx \dt \\
& \quad =  \an{\int_0^T} (h'(\vp) (\lambda_c \si - w_3^*) r, \xi) - (h_2 k(\vp), p)- (\chi \psi, q)  \dt \\
& \qquad + \an{\int_0^T}   (\lambda_p f(\vp) \psi g(\W_{,\E}(\vp, \E(\u))), p) \dt.
\end{align*}
Combining these two leads to the simplification
\begin{align*}
D\J_1(\bm{w}^*)[\bm{h}] & = \int_0^T \kappa (h_1, r)_\Gamma - h_2(k(\vp), p) + h_3 (h(\vp), r) \dt \\
& \quad  + \int_0^T  \gamma_1 (w_1^*, h_1)_\Gamma \dt + \int_0^T \gamma_2 w_2^* h_2 + \gamma_3 w_3^* h_3 \dt,
\end{align*}
and \eqref{nes:opt} is then a consequence of \eqref{opt:2}.

\section{Sparsity of non-negative optimal controls}\label{sec:spar}
In the medical context, the control variables $w_2 = m(t)$ and $w_3 = s(t)$ should be non-negative, and so we modify the admissible control subsets $\U_{ad}^{(2)}$ and $\U_{ad}^{(3)}$ to the following
\begin{align}\label{new:add}
\U_{ad}^{(i)} = \{ w_i \in L^\infty(0,T)\, : \, 0 \leq w_i(t) \leq \overline{w} \text{ for a.e.~} t \in (0,T)\} \text{ for } i = 2,3,
\end{align}
where $\overline{w}$ is a fixed positive constant.  If $\gamma_1 > 0$, from the optimality condition \eqref{nes:opt}, substituting $y_2 = w_2^*$ and $y_3 = w_3^*$, then using the Hilbert projection theorem allows us to infer that $w_1^*$ is the $L^2(\Sigma)$-orthogonal projection of $-\kappa r/ \gamma_1$ onto the closed and convex subset $\U_{ad}^{(1)}$ of $L^2(\Sigma)$, leading to the representation formula
\begin{align*}
w_1^*(x,t) = \min \big ( \overline{w}_1(x,t), \, \max \big ( \underline{w}_1(x,t), \, -\tfrac{\anold{\kappa}}{\gamma_1} r(x,t) \big ) \, \big ) \text{ for a.e.~}(x,t) \in \Sigma.
\end{align*}
In a similar fashion, if $\gamma_2, \gamma_4 > 0$, substituting $y_1 = w_1^*$ and $y_3 = w_3^*$ in \eqref{nes:opt} leads to representation formula
\begin{align}\label{rep:2}
w_2^*(t) = \PP_{[0, \overline{w}]} \left ( \frac{1}{\gamma_2} \left(\int_\Omega k(\vp(x,t)) p(x,t) \dx - \gamma_4 \lambda_2(t) \right ) \right ) \text{ for a.e.~} t \in (0,T),
\end{align}
where $\PP_{[a,b]} : \R \to [a,b]$ denotes the pointwise projection function
\begin{align*}
\PP_{[a,b]}(s) = \min (b, \, \max(a, s)).
\end{align*}
\anold{Similarly,} if $\gamma_3, \gamma_5 > 0$, then substituting $y_1 = w_1^*$ and $y_2 = w_2^*$ in \eqref{nes:opt} leads to representation formula
\begin{align*}
w_3^*(t) = \PP_{[0, \overline{w}]} \left ( \anold{-\frac{1}{\gamma_3}} \left(\an{\gamma_5 \lambda_3(t) + \int_\Omega h(\vp(x,t)) r(x,t) \dx} \right ) \right ) \text{ for a.e.~} t \in (0,T).
\end{align*}
Due to the $L^1$-regularisation for $w_2$ and $w_3$ in the optimal control problem, we can expect the optimal controls $w_2^*$ and $w_3^*$ to vanish on certain parts of the time interval $[0,T]$.  This is formulated as follows.

\begin{thm}\label{thm:spas}
Under \eqref{ass:beta}-\eqref{ass:target}, let $\bm{w}^* = (w_1^*, w_2^*, w_3^*) \in \U_{ad}$ be an optimal control where $\U_{ad}^{(2)}$ and $\U_{ad}^{(3)}$ are now given as \eqref{new:add} \anold{with the associated state $(\vp, \mu, \si, \u) = \bS(\bm{w}^*)$ and adjoint variables $(p,q,r,\bs)$.}  Then, we have the following characterisations:
\begin{itemize}
\item \an{If} $\gamma_2, \gamma_4 > 0$, for a.e.~$t \in (0,T)$,
\begin{align}\label{spa:1}
w_2^*(t) = 0 \quad \Longleftrightarrow \quad \int_\Omega k(\vp(x,t)) p(x,t) \dx \leq \gamma_4.
\end{align}
\item \an{If} $\gamma_3, \gamma_5 > 0$, for a.e.~$t \in (0,T)$,
\begin{align}\label{spa:2}
w_3^*(t) = 0 \quad \Longleftrightarrow \quad \int_\Omega h(\vp(x,t)) r(x,t) \dx \geq -  \gamma_5.
\end{align}
\end{itemize}
\end{thm}

\begin{remark}
Similar one-sided inequalities characterising sparsity of optimal controls are common when the lower bound in the admissible control sets $\U_{ad}^{(i)}$ is zero, \an{see, e.g.,}\cite[Thm.~3.1]{CasN} or \cite[Thm.~3.3]{Cas}.  If we allow $\underline{w}_2(t), \underline{w}_3(t)$ in the definition \eqref{Uad} of $\U_{ad}^{(2)}$ and $\U_{ad}^{(3)}$ to be a negative constant $\underline{w}$, then it is possible to provide a representation formula also for $\lambda_2$ and $\lambda_3$, see \cite{ST} in the context of Cahn--Hilliard tumour models, and also \cite{CasN,Cas,Stadler} for classical parabolic and elliptic control problems.  However, a negative lower bound for the controls may not be applicable in a medical context.
\end{remark}

\begin{remark}
We point out that if $w_2^*(\anold{t_0}) = 0$ for some $\anold{t_0} \in (0,T)$, then there exists an open subinterval $I \subset (0,T)$ with $\anold{t_0} \in I$ such that $w_2^*(t) = 0$ for all $t \in I$.  The same assertion also holds for $w_3^*$ provided $\beta > 0$.  This is due to the fact that the mappings
\begin{align*}
t & \mapsto \int_\Omega k(\vp(x,t)) p(x,t) \dx, \\
t & \mapsto \int_\Omega h(\vp(x,t)) r(x,t) \dx \text{ if } \beta > 0,
\end{align*}
are continuous in light of the regularities $\vp \in C^0([0,T];L^2(\Omega))$ from Theorem \ref{thm:state}, $p \in C^0([0,T];L^2(\Omega))$ and $r \in C^0([0,T];L^2(\Omega))$ if $\beta > 0$ from Theorem \ref{thm:adj}.  In particular, this behaviour where the optimal controls are zero over an interval is consistent with the prevailing medical practice in which there are periods in the overall treatment where radiation/cytotoxic therapies are not applied to patients.\end{remark}

\begin{proof}
Let us just present the details for \eqref{spa:1}, as \eqref{spa:2} can be derived in an analogous manner.  Due to the modification to $\U_{ad}^{(2)}$, we observe from \eqref{lambd} that $\lambda_2(t) \in L^\infty(0,T)$ satisfies
\begin{align}\label{new:lambd}
\lambda_2(t) \in \begin{cases} \{1 \} & \text{ if } w_2^*(t) > 0, \\
[-1,1] & \text{ if } w_2^*(t) = 0,
\end{cases}
\end{align}
for a.e.~$t \in (0,T)$.  The left implication proceeds as follows: Consider the set $E = \{ t \in (0,T) \, : \, w_2^*(t) = 0\}$, where by the representation formula \eqref{rep:2}, we see that 
\begin{align*}
\int_\Omega k(\vp(x,t)) p(x,t) \dx - \gamma_4 \lambda_2(t) \leq 0\quad  \text{ for all } t \in E.
\end{align*}
Using \eqref{new:lambd} and rearranging, we obtain the left implication of \eqref{spa:1}.  For the right implication we argue by contrapositive:  Suppose $w_2^*(t) > 0$, then from \eqref{new:lambd} we have $\lambda_2(t) = 1$ and thus by the representation formula it holds that
\begin{align*}
\int_\Omega k(\vp(x,t)) p(x,t) \dx - \gamma_4 \lambda_2(t) =  \int_\Omega k(\vp(x,t)) p(x,t) \dx - \gamma_4 > 0.
\end{align*}
Upon rearranging we obtain the assertion
\begin{align*}
w_2^*(t) > 0 \quad \implies \quad \int_\Omega k(\vp(x,t)) p(x,t) \dx > \gamma_4,
\end{align*}
which gives the right implication of \eqref{spa:1}.
\end{proof}

An interesting consequence is that we can identify $w_2^*(t) \equiv 0$ as a local optimal control provided $\gamma_4$ is sufficiently large, and similarly $w_3^*(t) \equiv 0$ is a local optimal control provided $\gamma_5$ is sufficiently large. 

\begin{cor}
\an{Suppose \eqref{ass:beta}-\eqref{ass:target} and $\gamma_2 > 0$. Then} there exists $\gamma_* > 0$ such that for $\gamma_4 > \gamma_*$, $w_2^*(t) \equiv 0$ for all $t \in (0,T)$ is an optimal control for \eqref{opt}.  Similarly, suppose $\gamma_3 > 0$ and $\beta > 0$, then there exists $\gamma^* > 0$ such that for $\gamma_5 > \gamma^*$, $w_3^*(t) \equiv 0$ for all $t \in (0,T)$ is an optimal control for \eqref{opt}.
\end{cor}

\begin{proof}
It suffices to use conditions \eqref{spa:1} and \eqref{spa:2}.  In light of the admissible control subsets defined in \eqref{new:add}, where $\overline{w}$ is a fixed constant, the constant $K_1$ in \eqref{state:bdd} is independent of the weights $\{\alpha_Q, \alpha_\Omega, \alpha_\E, \gamma_1, \gamma_2, \gamma_3, \gamma_4, \gamma_5\}$ in the optimal control problem \eqref{opt}.  Then, revisiting the proof of Theorem \ref{thm:adj}, we note that $\{\gamma_i\}_{i=1}^5$ do not appear in the adjoint system \eqref{adj}.  Consequently, the positive constants on the right-hand side of \eqref{adj:est1} are independent of $\{\gamma_i\}_{i=1}^5$.  Employing \eqref{ass:f} on the boundedness of $h$ and $k$, we deduce that 
\begin{equation*}
\begin{alignedat}{3}
\int_\Omega k(\vp(x,t)) p(x,t) \dx &\leq C_1 \quad &&\text{ for a.e. } t \in (0,T), \\
\int_\Omega h(\vp(x,t)) r(x,t) \dx & \geq - C_2 \quad && \text{ for a.e. } t \in (0,T) \text{ if } \beta > 0,
\end{alignedat}
\end{equation*}
for positive constants $C_1$ and $C_2$ independent of $\{\gamma_i\}_{i=1}^5$.  The assertion now follows from \eqref{spa:1} and \eqref{spa:2} by choosing $\gamma_* = C_1$ and $\gamma^* = C_2$.
\end{proof}

\section*{Acknowledgements}
The work of the second author is partially supported by a grant from the Research Grants Council of the Hong Kong Special Administrative Region, China [Project No.: HKBU 14302319].
The third author gratefully acknowledges financial support from the LIA-COPDESC initiative and from the research training group 2339 ``Interfaces, Complex Structures, and Singular Limits'' of the German Science Foundation (DFG).

\end{document}